%% file: locsolarxiv2.tex
\documentclass[draft,Gr,amssymb,12pt]{amsart}
\usepackage{amsmath,Gr,oldlfont,amssymb}
\normalsize
\numberwithin{equation}{section}
\newtheorem{thm}[equation]                {Theorem}
\newtheorem{lem}[equation]      {Lemma}
\newtheorem{prop}[equation]     {Proposition}
\newtheorem{cor}[equation]          {Corollary}

\theoremstyle{definition}
\newtheorem{definition}[equation]   {Definition}

\theoremstyle{remark}
\newtheorem{remark}[equation]       {Remark}

\begin{document}

\begin{titlepage}
\author{Christopher Winfield}
\email{winfield\@@madscitech.org}
\thanks{The author thanks Professor D. M\"uller for his discussions
and the Mathematics Department of CAU-Kiel for their kind hospitality during the early stages of this research.}
\pagestyle{myheadings} \markboth{Draft: C. Winfield}{Report}
\end{titlepage}
\begin{abstract}
\title{
Local Solvability on $\BH_1$: Non-homogeneous Operators
}
\begin{sloppypar}
Local solvability and non-solvability are classified for left-invariant differential operators
on the Heisenberg group $\Bbb{H}_1$ of the form $L=P_n(X,Y)+Q(X,Y)$ where the
$P_n$ are certain homogeneous polynomials of order $n\geq 2$ and $Q$ is of lower order with $X=
\partial_x,$ $Y=\partial_y+x\partial_w$ on $\BR^3$.
We extend previous studies of operators of the form $P_n(X,Y)$
via representations involving
ordinary differential operators with a parameter.
\end{sloppypar}
\end{abstract}
\maketitle
\pagenumbering{arabic}
\section{Introduction}
\input{newntroresub.tex}\label{s1}
\section{Canonical bases for ker$\cL_{\gm}$}\label{s2}
\input{ls1resub.tex}
\input{ls2resub.tex}

\input{ls3resub.tex}
\section{Further estimates: Analytic extensions, Wronskians and Adjoints}\label{s3}
\input{ls4resub.tex}
\section{Parametrices}\label{s4}
\input ls5resub.tex
\section{Solvability}\label{s5}
\input ls6resub.tex
\section{Non-solvability}\label{s6}
\input ls6bresub.tex
\input ls7resub.tex
\section{Some second-order examples}\label{s7}
\input ls8resub.tex
\section{Appendix}\label{appendix}
\input{appendixresub.tex}

\input lsbibresub.tex
\end{document}

%% file: newntroresub.tex
We continue our study of the solvability of operators of the form $P(X,Y)$ for
certain left-invariant vector fields on the Heisenberg group with underlying space
$\BR^3.$  We choose the realization of the corresponding Lie algebra using
$X\df$ $\partial_x,$ $Y\df$ $\partial_y+x\partial_w$ for (group) variables $x,y,w.$
Note that our operators include all generators of $\fh_1^{\BC}$ since
$\partial_w$ $=[X,Y].$ Local solvability for such a class of operators is well researched
and we defer to \cite{m1,m2,w1} and the references therein for an introduction to the present research.
Our work closely follows those techniques used in \cite{c,w1,w2,w3} to study related operators.

The study operators of order $n\geq 2$ which can be expressed as polynomials (in operator notation)
of the form
\begin{equation}L=P(X,Y)=\sum_{l=0}^nP_l(X,Y)\label{op} \end{equation}
where each $P_l$ is homogeneous of degree $l$ in the non-commuting variables $X,Y$ with complex (constant)
coefficients. Moreover, the highest-order terms $P_n$ form a so-called generic
operator by which we mean the following: In the complex variable $z,$
$P_n(iz,0)=z^n$ and $P_n(iz,1)$ has distinct complex (characteristic)
roots $\{\gamma_j\}_{j=1}^n$.
In this article we will characterize local solvability of operators $L$ in terms of
related ordinary differential operators of the form
$\cL_{\gm}^{\pm}$ $=$ $\sum_{l=0}^n\gm^{n-l}P_l(i\partial_t,\pm t)$ and $\cL_{\infty}^{\pm}$
$=$ $P_n(i\partial_t,\pm t)$ along their respective adjoint operators.

The major object of the present work is to examine the effect, if any,
that the inclusion of lower-order terms
has to the solvability of a homogenous left-invariant operator.
Some of the more famous results in local (non-)
solvability are characterized in terms of the principle symbol
defined on $T^*(\BR^m)$ given by $p_n(\vec{x},\vec{\xi})$ $\df$
$\sum_{|\ga|=n}a_{\ga}(\vec{x})(i\vec{\xi})^{\ga}$ for an operator
$L=$ $\sum_{|\ga|=n}a_{\ga}(\vec{x})\partial_{\vec{x}}^{\ga}$.
Necessary and sufficient conditions for operators of principle
type appear in the works \cite{nt1,nt2}.
More general criteria appear in \cite{ho1} (Theorem 6.1.1):
For $L$ to be locally solvable, $p_n(\vec{x},\vec{\xi})$ must satisfy
$p_n(\vec{x},\vec{\xi})$ $=0$ $\implies$
$\sum_{i=1}^m\partial_{{x_i}}p_n(\vec{x},\vec{\xi})\partial_{\xi_i}
\bar{p_n}(\vec{x},\vec{\xi})$ $-$ $\partial_{{\xi_i}}p_n(\vec{x},\vec{\xi})
\partial_{x_i}\bar{p_n}(\vec{x},\vec{\xi})$ $=0$.
In these results the highest-order derivative terms determine (non-)solvability; and, the inclusion of any smooth lower-order terms
do not alter this property.  The latter result has been applied
to various left-invariant operators on $\BH_m$ of various dimensions $m$ \cite{mpr}.
Our operators, however, are at least doubly characteristic (see \cite{w1}).
In contrast, our approach is, for the most part, to study solvability of operators $P(X,Y)$ as compared to
the solvability of the operator $P_n(X,Y)$, formed by highest-order terms of $P(X,Y)$ in the subalgebra of $\fh_1^{\BC}$ generated by
$X$ and $Y$ -
not necessarily those of highest order in differentiation: Note, for instance, that $\partial_w$ is of order two in the subalgebra, but has a symbol of order one.

We elaborate on our motivation for using the particular representations $\cL_{\gm}^{\pm}$ of $L$:
For $f(x,y,w)$ $\in$ $\cS(\BR^3)$ and $\check{\phantom{*}}$ ($\hat{{\phantom{*}}}$)   denoting
Fourier (inverse) transform with respect to the second and third variables,
we write $Lf(x,y,w)$ $=L(\hat{f}\,\check{)}(x,y,w)$ $=$
\begin{equation}\label{transform}
\frac{1}{2\pi}\int_{\BR^2}
e^{-i(\xi y+\eta w)}P(\partial_x,-i(\xi + \eta x))\hat{f}(x,\xi,\gm)d\xi d\eta\end{equation}
The change of variables $t= x\gm \pm \xi/\gm$ and $\gm=\sqrt{|\eta|}$
leads us to
the representations $\gL_{\gm}^{\pm}:$ $\gm>0$ given by $X\rightarrow \gm \dert,$ $Y\rightarrow \mp i{\gm}t$ (resp.)
whereby we obtain
\begin{align}\gL^{\pm}_{\gm}(L)&=\sum_{l=0}^n\gL_{\gm}(P_l(X,Y))
=(-i\gm)^n\sum_{l=0}^n\left(\frac{i}{\gm}\right)^{n-l}P_l\left(i\dert,t\right)\label{rep}\\
&=(-i\gm)^n\cL_{\gm}^{\pm}\,\,\,\,{\text \rm (resp.)}
\notag\end{align}
The realizations $\cL_{\gm}^{\pm}$ lend themselves to analysis as ordinary differential equations involving
a parameter (with singularity at $\gm=0$).
In turn, our positive results on solvability occur in the cases where solutions to $\cL_{\gm}^{\pm}f=g$
can
be used to construct parametrices for $L$. We introduce
\begin{definition} \label{def1} The operator
$\cL_{\gm}$
%$=\gm^{-n}P(-i\partial_t,-t)$
has as
regular parametrix on $\gO$ if for every bounded function $g$ $\in$ $\cC^{\infty}(\BR)$ there is
a function $F(t,\gm)$ satisfying the following:
\begin{itemize}
\item[1)] $F$ is a smooth function in the variables $(t,\gm)$ in domain $\gO;$
 \item[2)] $\cL_{\gm}F =g$ on $\gO;$ and,
 \item[3)] For every $m$ $\exists$ $a,C>0$ so that
 $|\partial_t^jF(t,\gm)|$ $\leq$ $C(1+|t|+|\gm|)^a$ on $\gO$
 for each $j:0\leq j\leq m.$
 \end{itemize}
\end{definition}
%We note that in the present work we cannot simply avoid singularities in our parametrices simply by
%judicious choice of $g$ as done in the aforementioned works.  Here, we will allow $t$ and $\gm$
%to be complex values in order to bypass singularities of our parametrices for $\gm$ in certain %neighborhoods
%of $0.$ So, we state
%\begin{definition}\label{def2}
% $\cL_{\gm}$ has a regular parametrix for $\gO$ $\subset$ $\BC^2$ if
%Definition \ref{def1} holds in the sense extended to complex derivatives and
%where $F$ and $g$ in analytic in their respective domains for $(t,\gm)\in$ $\gO.$
%\end{definition}

A general result which we are ready to present is
the following
\begin{lem}\label{lem1}
An operator $L$ as in (\ref{op}) is locally solvable if
each of the associated operators $\cL_{\gm}^{\pm}$ have regular
parametrices on $\BR\times (\gm_0,\infty)$ for some $\gm_0$ $>0.$
\end{lem}

From the above lemma we will obtain the following results:
\begin{thm}\label{thm1} For the operator $L$ as in (\ref{op})
suppose that the (generic) polynomial $P_n$
has characteristic roots $\gamma_j$ all with non-zero real parts. Then $L$ is locally solvable
if $P_n(X,Y)$ is locally solvable.
\end{thm}
From \cite{w1} we have immediately
\begin{cor}
The operator $L$ of Theorem \ref{thm1} is locally solvable if ker$(\cL^{\pm}_{\infty})^*$$\bigcap$$\cS(\BR)$ $=\{0\}$ for both choices
of $\pm$ sign.
\end{cor}

Results on non-solvability we are ready to state are as follows:
\begin{thm}\label{thm2}
Suppose $L$ is as in (\ref{op}) for some generic $P_n$. Then
$L$ is not locally solvable if, for some choice of $\pm$ sign, the set of parameters
$\gm\in\BR^+$: ker$(\cL^{\pm}_{\gm})^*$$\bigcap$$\cS(\BR)\setminus\{0\}$ $\neq\emptyset$ has a
limit point in $\BR^+.$
\end{thm}
\begin{thm} \label{thm3} An operator as in Theorem \ref{thm1} is not locally solvable if the cardinality
of either $\{\gamma_j|{\text Re}\gamma_j>0\}$ or $\{\gamma_j|{\text Re}\gamma_j<0\}$ is greater than $n/2.$
\end{thm}

The outline of the article is as follows: In Section \ref{s2} we establish estimates of bases
for ker$\cL_{\gm}$ for $(t,\gm)$ in real domains as in Lemma \ref{lem1}.  In Section \ref{s3} we
establish estimates with $(t,\gm)$ extended to certain complex domains. In Section \ref{s4}
we construct solutions to $Lu=f$ to prove Lemma \ref{lem1} and Theorem \ref{thm1}.
In Section \ref{s5} we provide particular results on non-solvability, including proofs of Theorems \ref{thm2} and \ref{thm3}.
In Section \ref{s6} we develop criteria for non-solvability by which we develop
exact conditions for some subclass of operators $L$. In Section \ref{s7} we apply our general
results to various classes of second-order operators and compare our results to those of a well-known subclass. 

%% file: ls1resub.tex
We will introduce some notation which we will use throughout the remainder of the
article: Given functions $f$ and $g,$ the expression $f\lesssim g$ will mean $\exists C>0$ (fixed) so that $|f|\leq C|g|$ holds
on the specified domain; and, the expression $f\asymp g$ will mean that $f\lesssim g$ and $g\lesssim f$ both hold.

We now derive asymptotic estimates for certain bases of
ker$\cL_{\gm}$ for real $\gm$ following a diagonalization procedure in \cite{c}.
We note that for each
$l$ $P_l(i\partdert,t)$ can be written in the form
$$P_l(i\partdert,t) = \sum_{j=0}^lt^{n-j}q_{l,j}(t)\partdertn{j}$$
for $q_{l,j}(t) = d_{l,j} + \tilde{q}_{l,j}(t)$ where the $d_{l,j}$
are complex constants (vanishing for $j>l$) and
$\tilde{q}_{l,j}(t)=\sum_{0<2m\leq n-j}\frac{e_{l,j,m}}{t^{2m}}$
for some complex constants $e_{l,j,m}$ (see equation (2.4) of
[C]); in fact,
$d_{n,n-1}=-i\sum_{j=0}^n\gc_j.$

By way of rearrangement, we write ${\cal L}_{\gm}=$
\begin{align}
 \sum_{l=0}^n\frac{1}{\gm^{n-l}}P_l(i\partdert,t)
&= \sum_{l=0}^n\gm^{-l}\sum_{j=0}^l
t^{l-j}q_{l,j}(t)\partdertn{j}
\notag\\
=\sum_{j=0}^n[\sum_{l=0}^{n-j}\frac{1}{\gm^{l}}
t^{n-l-j}q_{n-l,j}(t)]\partdertn{j}
&=\sum_{j=0}^n
t^{n-j}[\sum_{l=0}^{n-j}\frac{1}{\gm^{l}}
\frac{1}{t^{l}}
q_{n-l,j}(t)]
\partdertn{j}\notag\\
&=\sum_{j=0}^nt^{n-j}
Q_{j}(t,\gm)\partdertn{j}&
\notag
\end{align}
for
$Q_j(t,\gm)=q_{n,j}(t)+\sum_{l=1}^{n-j}
\frac{1}{(\gm t)^{l}}q_{n-l,j}(t)$
(a vacuous sum is taken to be zero).
Therefore,
$$Q_j(t,\gm) = d_{n,j} + \frac{d_{n-1,j}}{\gm t}+
\left(e_{n,j,1}+\frac{d_{n-2,j}}{\gm^2}\right)
\frac{1}{t^{2}}+\epsilon_j(t,\gm)$$ where
$\epsilon_j(t,\gm)$ may be expressed as a linear combination with
complex coefficients of monomials of the form
$\frac{1}{\gm^{a}t^{b}}$ for integers $a\geq 0$ and
$3\leq b\leq l$.  We note that, for $t$ restricted to any compact
subset of ${\Bbb R}$, $Q_j(t,\gm)$ converges to $q_j(t,\gm)$
uniformly as $\gm \rightarrow \infty.$

Let us
formulate the differential equation ${\cal L}_{\gm}y =0$ in terms
of an equivalent matrix equation in that ${\cal L}_{\gm}f=0$ if
and only if
$u^{\prime}=Au$ where
$u$ is the column vector
$(f,f^{\prime},\hdots,f^{(n-1)})^{\dagger},$ with
$\dagger$ denoting transpose, and where $A=$
$$\left( \begin{matrix}
0&1&0&\hdots&0\\
0&0&1&\hdots&0\\
\vdots&\vdots&\vdots&&\vdots\\
-t^nQ_0(t,\gm)&-t^{n-1}Q_1(t,\gm)&
-t^{n-2}Q_2(t,\gm)&\hdots&-tQ_{n-1}(t,\gm)
\end{matrix}
\right)
$$

We now seek to diagonalize $A$ modulo
appropriate error terms.
We single out the lower-order terms
of $A$ as we define its principal part
$A_0$ by
$$A_0=\left( \begin{matrix}
0&1&0&\hdots&0\\
0&0&1&\hdots&0\\
\vdots&\vdots&\vdots&&\vdots\\
-t^nd_0&-t^{n-1}d_1&
-t^{n-2}d_2&\hdots&-td_{n-1}
\end{matrix}
\right).
$$
We also define
$$S_0=\left( \begin{matrix}
1&1&1&\hdots&1\\
\gamma_1 t&\gamma_2 t&\gamma_3 t&\hdots&\gamma_n t\\
(\gamma_1 t)^2&(\gamma_2 t)^2&(\gamma_3 t)^2&\hdots&(\gamma_n t)^2\\
\vdots&\vdots&\vdots&&\vdots\\
(\gamma_1 t)^{n-1}&(\gamma_2 t)^{n-1}&(\gamma_3 t)^{n-1}
&\hdots&(\gamma_n t)^{n-1}\\
\end{matrix}
\right)
$$
and
$$\Lambda_0=\left( \begin{matrix}
\gamma_1t & 0 & 0 &\hdots&0\\
0&\gamma_2t&0 &\hdots &0\\
\vdots& \vdots&\vdots&\ddots&\vdots\\
0&0&0&\hdots&\gamma_nt
\end{matrix}
\right).
$$
We note that $S_0$ diagonalizes $A_0$ in that $A_0S_0=S_0\Lambda_0.$
We consider higher-order terms (in $t$) by setting (formally)
$A = A_0+\cE_1 +\cE_2 +\cE_3$
for
\begin{equation}\label{cE}
\cE_1 \df
\left( \begin{matrix}
0&0&\hdots&0\\
\vdots&\vdots&&\vdots\\
0&0&\hdots&0\\
a_1t^{n-1}&a_2t^{n-2}&\hdots&a_n1
\end{matrix}
\right),
\end{equation}
$$
\cE_2 \df
\left( \begin{matrix}
0&0&\hdots&0\\
\vdots&\vdots&&\vdots\\
0&0&\hdots&0\\
b_1t^{n-2}&b_2t^{n-3}&\hdots&b_{n}t^{-1}
\end{matrix}
\right)
$$
%($b_k=0$ for $k=n$) and
$$
\cE_3 \df
\left( \begin{matrix}
0&0&\hdots&0\\
\vdots&\vdots&&\vdots\\
0&0&\hdots&0\\
-\epsilon_1(t,\gm)t^{n-3}&-\epsilon_2(t,\gm)t^{n-4}
 &\hdots&-\epsilon_n(t,\gm)t^{-2}
\end{matrix}
\right)
$$
($\epsilon_k=0$ for $k\geq n-2$). Here the non-vanishing
coefficients are given by $a_j$ $=$ $\frac{-d_{n-1,j-1}}{\gm},$
$b_j$ $=$ $-(e_{n,j,1}+\frac{d_{n-2,j-2}}{\gm^2})$ and
$\epsilon_j(t,\gm)$ as above for $1\leq j\leq n$.

Regarding $S_0$ $=S_0(t)$ as a matrix-valued function
of $t$, it is not difficult to show
$$\det S_0(t) = t^{\frac{n(n-1)}{2}}\det S_0(1).
$$
Since $S_0(1)$ is a VanderMonde Matrix,
$S_0(t)$ is invertible for all $t\neq 0:$
Indeed,
$$[S_0^{-1}(t)]_{i,j} = \frac{1}{t^{j-1}}
[S_0^{-1}(1)]_{i,j}.
$$

We define
$S=S_0(I+ {\cal A}+\Delta )$
where
${\cal A}\df \frac{1}{t}(\ga_{i,j})_{1\leq i,j\leq n}\,\text{\rm and}\,
\Delta\df \frac{1}{t^2}(\gd_{i,j})_{1\leq i,j\leq n}$
for complex constants $\ga_{i,j}$ and $\gd_{i,j}$ to be
determined: We set the diagonal elements $\ga_{j,j}=\gd_{j,j}=0$.
Formally, we write
\begin{align}
S^{-1}&=(I +\cA+\gD)^{-1}S_0^{-1}\notag\\
&=I-(\cA+\gD)(I+\cA+\gD)^{-1}\notag\\
&=I-(\cA+\gD)+(\cA+\gD)^2(I+\cA+\gD)^{-1}\notag
\end{align}

To appraise the error terms in the diagonalization
$S^{-1}AS$ we introduce the notation
$A=O(t^p)$, meaning that all entries of the matrix
$A$ are majorized by $t^p$ uniformly for all sufficiently
large $t$ and $\gm.$
Let us define
$\cD_i\df S_0^{-1}\cE_iS_0$ for $i=1,2,3.$ Since $\cE_1S_0=$
$$
\left( \begin{matrix}
0&0&\hdots&0\\
\vdots&\vdots&&\vdots\\
0&0&\hdots&0\\
\sum_{j=1}^na_j\gamma_1^{j-1}t^{n-1}&\sum_{j=1}^na_j\gamma_2^{j-1}t^{n-1}
&\hdots&\sum_{j=1}^na_j\gamma_n^{j-1}t^{n-1}
\end{matrix}
\right)
$$
we see that $[\cD_1]_{i,j}$, the $i,j$th element of
$\cD_1$, is given by
\begin{equation}
[\cD_1]_{i,j} = [S_0^{-1}(t)]_{i,n}\sum_{k=1}^n
a_{k}\gamma_j^{k-1}t^{n-1}
=[S_0^{-1}(1)]_{i,n}\sum_{k=1}^n
a_{k}\gamma_j^{k-1}
\label{D1}\end{equation}
so that $\cD_1=O(1)$ is constant
with respect to $t$; in fact, $\cD_1$ $\rightarrow$
$0^{n\times n}$ as $\gm$ $\rightarrow$ $\infty$.

Likewise, $\cE_2S_0=$ $$ \left( \begin{matrix}
0&0&\hdots&0\\
\vdots&\vdots&&\vdots\\
0&0&\hdots&0\\
\sum_{j=1}^nb_j\gamma_1^{j-1}t^{n-2}&\sum_{j=1}^nb_j\gamma_2^{j-1}t^{n-2}
&\hdots&\sum_{j=1}^nb_j\gamma_n^{j-1}t^{n-2}
\end{matrix}
\right)
$$
and
$$
[\cD_2]_{i,j} = [S_0^{-1}(t)]_{i,n}(\sum_{k=1}^n
b_k\gamma_j^{k-1})t^{n-2}
=\frac{1}{t}[S_0^{-1}(1)]_{i,n}\sum_{k=1}^n b_k\gamma_j^{k-1} \notag
$$
so that $\cD_2=O({t^{-1}}).$ It follows
similarly that $\cD_3 = O({t^{-2}}).$

We proceed with diagonalization of  $A$ as we compute
\begin{align}
&S^{-1}AS
\notag\\
=&[I -(\cA +\gD)+(\cA^2+\cA\gD+\gD\cA+\gD^2)(I+\cA+\gD)^{-1}]S_0^{-1}
\notag\\
\circ
& (A_0 + \cE_1+\cE_2+\cE_3)S_0[I+\cA+\gD]
\notag\\
=&
[I -(\cA +\gD)+(\cA^2+\cA\gD+\gD\cA+\gD^2)(I+\cA+\gD)^{-1}]
\notag\\
\circ
&(\Lambda_0+\cD_1+\cD_2+\cD_3)[I+\cA+\gD]
\label{SinvAS}
\end{align}

Given fixed $\ga_{i,j}$ and $\gd_{i,j}$, the
matrix $I+\cA+\gD$ is invertible for all sufficiently
large $t$ for (say) $\gm\geq 1$ and $(I+\cA+\gD)^{-1}$ $=$
$O(1)$.
Upon multiplying out (\ref{SinvAS}) and collecting terms
up to $O(t^{-2}),$ one obtains
\begin{align}
S^{-1}AS=&
\Lambda_0\label{mult1a}\\
+&\cD_1 - [\cA,\Lambda_0]\label{mult1b}\\
+&\cD_2-[\cA,\cD_1]+\cA[\cA,\gL_0]-[\gD,\gL_0]\label{mult1c}\\
+&O(t^{-2})\notag
\end{align}

The elements of (\ref{mult1a}) are known from
above; each summand in (\ref{mult1b})
is $O(1)$, indeed constant w.r.t. $t$; and, each summand in (\ref{mult1c})
is $O({t^{-1}})$ and, more precisely,
is a product of $\frac{1}{t}$ times a matrix
that is constant with respect to $t.$

We make the substitution
$v=S^{-1}u$ so that the differential equation
becomes
$v^{\prime}=Bv$
for
$B=S^{-1}AS-S^{-1}S^{\prime}.$
We first estimate $S^{-1}S^{\prime}$ according to
\begin{align} S^{-1}S^{\prime}&=
S_0^{-1}S_0^{\prime}-(I+\cA+\gD)^{-1}(\cA+\gD)S_0^{-1}S_0^{\prime}\notag\\
+&S^{-1}S^{\prime}\gD +S^{-1}S_0^{\prime}\cA
+S^{-1}S_0(\cA^{\prime}+\gD^{\prime})\notag.
\end{align}
To estimate $S^{-1}S^{\prime}$ modulo terms of order $O(t^{-2})$, we
find that it suffices to estimate $S_0^{-1}S_0^{\prime}:$ We have
$$
S_0^{\prime}=
\left( \begin{matrix}
0&0&\hdots&0\\
\gamma_1&\gamma_2&\hdots&\gamma_n\\
\vdots&\vdots&&\vdots\\
(n-1)\gamma_1^{n-1}t^{n-2}&(n-1)\gamma_2^{n-1}t^{n-2}&
\hdots&(n-1)\gamma_n^{n-1}t^{n-2}\\
\end{matrix}
\right);
$$
so, the $i,j$th element of $S_0^{-1}S_0^{\prime}$ is given by
\begin{align}
[S_0^{-1}S_0^{\prime}]_{i,j} &= \sum_{k=1}^n[S_0^{-1}(t)]_{i,k}
(k-1)\gamma_j^{k-1}t^{k-2}
\notag\\
&=
\frac{1}{t}
\sum_{k=1}^n[S_0^{-1}(1)]_{i,k}
(k-1)\gamma_j^{k-1}
\notag
\end{align}
and $S_0^{-1}S_0^{\prime}$ $=$ $O({t^{-1}}).$
Now,
\begin{align}
S^{-1}S_0^{\prime}
&=(I+\cA+\gD)^{-1}S_0^{-1}S_0^{\prime}
\notag\\
&=(I-(\cA+\gD)(I+\cA+\gD)^{-1})S_0^{-1}S_0^{\prime}
\notag\\
&=S_0^{-1}S_0^{\prime}
-
(\cA+\gD)(I+\cA+\gD)^{-1}S_0^{-1}S_0^{\prime}.
\notag
\end{align}
It is now clear that
$S^{-1}S^{\prime}-S^{-1}_0S_0^{\prime}= O(\frac{1}{t^2}).$
Therefore,
\begin{align}
&S^{-1}AS-S^{-1}S^{\prime}=\notag\\
&\gL_0
\\
+&\cD_1
-[\cA,\gL_0]
\\
+
&\cD_2 +[\cD_1,\cA] +
\cA[\cA,\gL_0] -S_0^{-1}S_0^{\prime} -[\gD,\gL_0]
\\
+&O(t^{-2}).\notag
\end{align}

Since the roots $\gamma_j$ are distinct we may define the elements
of $\cA$ uniquely by setting its diagonal elements $\gd_{j,j}$ to
zero and by setting
$$[\cA,\gL_0]=\cD_1 \,\,\text{\rm off\,\,the\,\,diagonal.}$$
We then set
$$[\gD,\gL_0]=\cD_2+[\cD_1,\cA]+\cA[\cA,\gL_0]+S_0^{-1}S_0^{\prime}
 \,\,\text{\rm off\,the\,diagonal}
$$
and set the diagonal elements $\ga_{j,j}$ to zero
to define the matrix $\gD$. We now set $\forall j$
\begin{equation}\gb_{\gm,j}\df
%\frac{\gb_{j}}{\gm}\df[\cD_1-[\cA,\Lambda_0]]_{j,j}=
[\cD_1]_{j,j}\label{beta}\end{equation}
and
$$\frac{\rho_{\gm,j}}{t}\df
[\cD_2 +[\cD_1,\cA]+\cA[\cA,\gL_0]-S_0^{-1}S_0^{\prime}
%-[\gD,\gL_0]
]_{j,j}.$$ We note that the $\gb_{j}$'s are complex constants
depending only on $P_n$ and $P_{n-1}$, specifically on the coefficients
$d_{n-1,k}$ and the roots $\gamma_k.$ Clearly, $\gb_{\gm,j}$
$\rightarrow 0$ as $\gm$ $\rightarrow \infty$ and, moreover,
$\rho_{\gm,j} \rightarrow$ $\rho_j$ as $\gm$ $\rightarrow$ $\infty$
$\forall j$
where the limits $\rho_j$ depend only on $P_n$

We now have
$B=\Lambda + \cR$
for
$$
\gL
\df
\left( \begin{matrix}
\gamma_1t+\gb_{\gm,1}+\frac{\rho_{\gm,1}}{t}&0&\hdots&0\\
0&\gamma_2t+\gb_{\gm,2}+\frac{\rho_{\gm,2}}{t}
&\hdots&0\\
\vdots&\vdots&\ddots&\vdots\\
0&0&\hdots&\gamma_nt+\gb_{\gm,n}+\frac{\rho_{\gm,n}}{t}
\end{matrix}
\right)
$$
where the matrix $\cR$ satisfies
$\cR=O({t^{-2}}).$

%% file: ls2resub.tex
We now define
\begin{equation}\label{gTh}
\Phi_j(t,\gm) \df \gamma_j\frac{t^2}{2} +\gb_{\gm,j}t+\rho_{\gm,j}\ln|t|.
\end{equation}
Since the coefficients $\gb_{\gm,j}$ and $\rho_{\gm,j}$ depend
only on $\gm$ and have definite limits as $\gm$ $\rightarrow$ $+\infty,$
we may rearrange the rows of $B$, if
necessary, to suppose that $\exists \gm_0,t_1$ $\geq 0$ so that
Re$\Phi_j(t,\gm)-$Re$\Phi_{j+1}(t,\gm)$ $\geq0$ for $t>t_1$ and $\gm>\gm_0$.

For $v$ as above, we define $w=e^{-\Phi_1}v:$ So, $w$ satisfies
\begin{align}
w^{\prime}
&=
-\Phi^{\prime}_1e^{-\Phi_1}v +
e^{-\Phi_1}v^{\prime}
\notag\\
&=
-\Phi^{\prime}_1e^{-\Phi_1}v +
e^{-\Phi_1}(\Lambda+\cR)
\notag
\\
&=(-\Phi^{\prime}_1I+\Lambda)w+\cR w
\label{w}
\\
&=
\tilde{\Lambda}w +\cR w
\notag
\end{align}
where $\tilde{\Lambda}$ is a diagonal matrix
with elements
\begin{equation}\label{exponsA}[\tilde{\Lambda}]_{j,j}=
-(\gamma_1-\gamma_j)t -(\gb_{\gm,1}-\gb_{\gm,j})
-(\rho_{\gm,1}-\rho_{\gm,j})\frac{1}{t}. \end{equation}

We now determine estimates for $w$ for sufficiently
large positive $t$, estimates for negative $t$ will be similar.
Choose $t>t_1$ so large that the estimates of $A$ hold and that
\begin{equation}\label{exponsB}{\text{\rm Re}}(\Phi_1(t,\gm)-\Phi_j(t,\gm))\geq 0\end{equation}
$\forall t\geq 0$ and $\forall \gm$
$\geq \gm_0$ for
$j>1.$

We choose $w(t;y)$ to be the unique solution such that
$w(y;y)$ $=$ $(0,\hdots,0,1)^{\dagger}$
for $y\geq x_1.$
It follows as in \cite{c} that
$$\dert|w(t,y)|^2\geq -2|\cR(t)||w(t,y)|^2$$
for $t\geq t_1$ so that
\begin{equation}
|w(t;y)|\lesssim  \exp(\int_t^y2|\cR(s)|\,ds)
\cdot|w(y;y)|
\lesssim 1.
\label{boot1}
\end{equation}
uniformly for $t,y\geq x_1$ and for
$\gm\geq \gm_0$.
Further, we then bootstrap as we apply (\ref{boot1})
to obtain
more accurate estimates for $w_j(t,y).$
For $j<n$
\begin{equation}
w^{\prime}_j(t;y)=(\Phi_j^{\prime}(t,u)-\Phi_1^{\prime}(t,u))
w_j(t;y) +O(|\cR(t)||w(t;y)|)
\label{boot2}
\end{equation}
so that
$$|w_j(t)|\lesssim \int_t^y
|\exp(\Phi_1(t,\gm)-\Phi_j(t,\gm)-(
\Phi_1(s,\gm)-\Phi_j(s,\gm)))|
|\cR(s)|ds
$$
Since the assignment
$t\rightarrow
\Re(\Phi_1(t,\gm)-\Phi_j(t,\gm))$
is an increasing function for $\gm>0,$
we have
\begin{equation}
|w_j(t,\gm)|
\lesssim\int_t^y|\cR(s)|\,ds
\lesssim  \frac{1}{t}
\notag
\end{equation}
uniformly in $t$ and $\gm$.
Also, from (\ref{boot1}) and (\ref{boot2}) we have
$$|w_n^{\prime}(t;y)|
\lesssim |\cR(t)||w(t;y)|
$$
so that
$$
|w_n(t,y)-1|\lesssim
\int_t^y |\cR(s)|\,ds
\lesssim \frac{1}{t}
$$

We now address the convergence of the function $w(\cdot,y),$ as
$y$ $\rightarrow$ $\infty.$ For $x_1\leq$ $t\leq$ $y_1\leq$ $y_2$
the function $w(t;y_1)-w(t;y_2)$ is a solution of equation (\ref{w}); and,
hence, (\ref{boot1}) and (\ref{boot2}) hold for $w(t;y)$ replaced
by $w(t;y_1)-w(t;y_2)$ so that
\begin{align}
|w(t;y_1)-w(t;y_2)| \lesssim& \int_{t}^{y_1}|\cR(s)|\,ds
\cdot|w(y_1;y_1)-w(y_1;y_2)|
\notag\\
\lesssim &
|w_n(y_1;y_2)-1|+\sum_{j=1}^{n-1}|w_j(y_1,y_2)|
\lesssim
\frac{1}{y_1}
\notag
\end{align}
uniformly in $t,$ $y_1$ and $y_2$. So, $w(t;y)$ converges uniformly
on compact subsets of $[t_1,\infty)$ as $y$ $\rightarrow$ $\infty$
to a function $w(t)$ which satisfies
$$|w_n(t)-1|+\sum_{j=1}^{n-1}|w_j(t)|\lesssim\frac{1}{t}$$
and satisfies
$\lim_{t\rightarrow \infty}w(t)$ $=$
$(0,\hdots,0,1)^{\dagger}.$
We have shown that for $v=e^{\Phi_1}w,$
\begin{equation}\label{v}
|v(t)|\lesssim |e^{\Phi_1}|;\,\,
|v_n(t)-e^{\Phi_1}|\lesssim  |e^{\Phi_1}|t^{-1};\,\,
|v_j(t)|\lesssim  |e^{\Phi_1}|t^{-1}
\,\,\text{\rm for}\ j<n:
\end{equation}
These estimates hold for all $t\geq t_0$ and $\gm \geq \gm_0$.

From $u=Sv$ we obtain a solution
to $\cL_{\gm}f=0$
from the component
$\psi_1(t,\gm)$ $\df$ $u_1$ with
$u_j=\dertn{j-1}\psi_1$
for $1\leq$ $j\leq$ $n.$
Moreover,
\begin{align}
u=&Sv
\notag
\\
=&S_0(I+\cA + \gD)v
\notag\\
=&
e^{\Phi_1}\cdot[
S_0\left(\begin{matrix}
0\\\vdots\\0\\1\end{matrix}\right)
+O(\frac{1}{t})].
\notag
\end{align}
We therefore have $\forall j\leq n-1$ that
$$\dertn{j}\psi_1(x,\gm) =(\gamma_1t)^je^{\Phi_1(t,\gm)}(1+o(1))$$
as $t\rightarrow$ $\infty$ where constants implicit in the estimates
hold uniformly for $\gm$ $\geq$ $\gm_0$. Moreover, it follows from
the construction that $\psi_1$ is of class $\cC^{\infty}(\BR)$ as a
function of $t$ for $t\geq t_0$ and holomorphic as a function of
$\gm$ for $\gm$ $\geq\gm_0.$ Then, $\psi_1$ extends to a function
$\cC^{\infty}(\BR)$ in $t$ and holomorphic in $\gm.$

%% file: ls3resub.tex
The above procedure may be carried out inductively and does not involve much alteration of
the cited work. Hence, we defer the remainder
of the proof of the following result to the Appendix:
\begin{prop}\label{moremainests}
There are bases $\{\psi_k^{\pm}(t,\gm)\}_{k=1}^n$ of $\text{\rm
ker}\cL_{\gm}$ of functions, $\cC^{\infty}(\BR)$ as functions of $t$
and holomorphic as functions of $\gm$ on Re $\gm>0,$ which for each $1\leq$ $k\leq
n$ and $0\leq j$ satisfy
\begin{equation}
\dertn{j}\psi^{\pm}_k(t,\gm)\lesssim(1+|t|)^j
e^{\Phi_k^{\pm}(t,\gm)}\label{mainest}
\end{equation}
for $\pm t>0$ (resp.) for $\Phi_k^{\pm}$'s as in
(\ref{gTh}).
\end{prop}
For ease of reference, we end this section with the following remark which follows
by inspection from equations (\ref{D1}) and (\ref{beta}):
\begin{remark}\label{remark1}
The transformation $P(i\partdert,t)$ $\rightarrow$
$P(i\partdert,-t)$ leads to the transformations $\gamma_j$
$\rightarrow$ $-\gamma_j$ and $d_{n-1,k}$ $\rightarrow$
$(-1)^{n-k}d_{n-1,k}.$
\end{remark}

%% file: ls4resub.tex
We start with the more novel techniques beyond those
of \cite{w1,w2,w3}. In those articles, as in this one,
singularities in the parametrices are a concern as the Fourier parameter $\eta$ tends to $0$.
However, in the previous articles the singularities
were canceled by solving
$P_n(i\dert,t)f=g$ for functions $g$ with zeros of appropriate order so as to
cancel such singularities in $f.$
Here, we cannot avoid the singularity of $\cL_{\gm}$ at $\gm=0$ by such procedures; but, we can bypass it
as we pass the parameter $\gm$ to the complex plain.

We now form estimates of bases of ker$\cL_{\gm}$ but now with
complex-valued $t,$ $\gm$. We note that the solutions can
be analytically continued \cite{cl} and, as we shall show, the estimates as in
Proposition \ref{moremainests} hold if (\ref{exponsB}) holds, perhaps for some
reordering of the characteristic roots $\gamma_j$. Such
estimates turn out to hold for $t$ contained in certain complex sets of
the form $\cK+S$ where $\cK$ is compact and $S$ is conic.
\begin{prop}\label{complex}
Let $\gO$ be a simply connected, compact subset of $\BC$ which
is a positive distance from the origin. For any given $\ga$
$\in$ $[0,\pi/2]$ there is a compact set $\cK$ $\subset$ $\BC$ where the operator ker$\cL_{\gm}$ has bases $\{\psi_j^{\pm}(t,\gm)\}_{j=1}^n$
for which estimates as in (\ref{mainest}) hold for $\gm\in\gO$ and $t$ on a
complex set of the form
\begin{equation}\label{sector} \fU_{\ga}^{\pm}=
\{z_0+se^{-i\theta}| \pm s>0,\gu_1<\gth-\ga
<\gu_2,z_0\in\cK\}\,\,\text{(resp.).}
\end{equation}

\end{prop} We note that the choice of $\cK$ may
require that Re $z_0e^{-i\ga}$ and Im $z_0e^{-i\ga}$ are sufficiently large and positive; and, our choices
of $\gu_1,\gu_2$ may depend only on $\ga.$
Furthermore, we may adjust the domains
$\cU^{\pm}_{\ga}$ to obtain $\cK,\gO,\gu_1$ and $\gu_2$ common to
each $\ga$ and $\pm$ sign.
\begin{proof} We will first prove the case for $\ga=0$. Let us
(re)arrange the characteristic roots so that the following hold
for each pair of indices $(j,l):$ $1\leq j<l\leq n $:
\begin{itemize}
\item[1)]
Re$(\gamma_j-\gamma_l)\geq 0,$
\item[2)] Re$(\gamma_j-\gamma_l)=0$ $\implies$ Im$(\gamma_j-\gamma_l)>0$
\end{itemize}
Then set $\gD_{j,l}\df\gamma_j-\gamma_l,$
$\gb_{j,l}$ $\df$ $\gb_j-\gb_l,$ and $\rho_{j,l}$
$\df\rho_j-\rho_l$ defined as in Proposition \ref{mainest}, but
with $\gamma_j$'s in the present arrangement.

We now set
$$\Phi_{j,l}(z,\gm)\df\gD_{j,l}z^2/2+\gb_{j,l}z/\gm+\rho_{j,l}/z$$
and compute ${\text{\rm Re}}\Phi_{j,l}(z_0+se^{-i\gth},\gm)=$
\begin{align}
&\frac{s^2}{2}(
{\text{Re}}\gD_{j,l}\cos(2\gth)
+{\text{Im}}\gD_{j,l}\sin(2\gth))\notag\\
&+s(\cos\gth{\text{Re}}(z_0\gD_{j,l}+\frac{\gb_{j,l}}{\gm})
+\sin\gth{\text{Im}}(z_0\gD_{j,l}+\frac{\gb_{j,l}}{\gm}))\notag\\
&+{\text{Re}}(z_0^2\gD_{j,l}+z_0\gb_{j,l}/\gm)
+{\text{Re}}(\rho_{j,l}/(z_0+se^{-i\gth}))\notag
\end{align}

We estimate $w(z)$ defined as in Proposition
\ref{mainest} but extended to complex domains:
There are positive $\gth_1,\gth_2,s_0$
so that
$$-\Phi_{j,l}(z_0+se^{-i\theta})+\Phi_{j,l}(z_0+s_0e^{-i\theta})<-c_2s+c_3$$ for some positive
constants $C,$ $c_1,c_2$ uniformly for $-\gth_1<$
$\gth<\gth_2$ and $s\geq s_0.$
Likewise to (\ref{w}) we find that the function
$w$
satisfies
$$|w(z_0+se^{-i\gth})|\leq C_{\gth}|w(z_0+se^{-i\gth})|$$
for some constant $C_{\gth}$ (depending on $\gth$) for each such $\gth.$
It now follows from an application
of the Phragmen-Lindel\"{o}f Theorem (cf.
Section 5.5
\cite{cl}) that, for a possibly smaller region
$-\gth_1\leq$ $\gu_1<\gth$ $<\gu_2$ $\leq$ $\gth_2$,
$w(z)$ is bounded.
The various induction arguments as in Proposition \ref{moremainests} then follow
to complete the proof in this case.

For $\ga\neq 0$ we set $\zeta=te^{-i\ga}$ and apply the above
arguments to
$$e^{ni\ga}P(ie^{-i\ga}\partial_{\zeta},\zeta e^{i\ga})
=\sum_{l=0}^n\left(\frac{e^{i\ga}}{\gm}\right)^{n-l}P_l(i\partial_{\zeta},\zeta e^{2i\cdot\ga}).
$$
We replace the characteristic roots $\gamma_j$ by
$\tilde{\gamma_j}$$=\gamma_je^{2i\ga}$ $\forall j$ and
rearrange them so that Re$(\tilde{\gamma_j}$
$-\tilde{\gamma_l}),$ Im$(\tilde{\gamma_j}$
$-\tilde{\gamma_l})$ $\geq 0$ again for indices $1\leq j<l\leq
n$.
\end{proof}

Using the variation of constants formula to form solutions to
$\cL_{\gm} f=g$ \cite{cl}, we need to analyze certain Wronskians and related determinants:
Given a basis $\vec{\psi}$ of ker$\cL_{\gm}$ we set
$W(\vec{\phi})(t,\gm)$ $=$ $\det A$ with the $n\times n$ matrix $A$ given by $[A]_{k,j}=\dertn{k-1}\phi_j$
and $W_l(t,\gm)$ defined likewise but with the $l$-th column of $A$ replaced by
$(0,\hdots,0,1)^{\dagger}.$ We will apply superscript $\pm$ to the $W_l$'s and to $W$ to indicate
their corresponding basis pairs $\vec{\psi^{\pm}};$ or, we may
simply drop the superscript when the basis is clearly implied.
\begin{prop}\label{thehs}
Suppose $\vec{\psi}^{\pm}$ is a basis of $\cL_{\gm}$ as in
Proposition \ref{moremainests} defined for $(t,\gm)$ in
$\BR^{\pm}\times(\gm_0,\infty).$ Then, for some real constant $a,$
the functions $h_j^{\pm}\df$ $W_j^{\pm}/W^{\pm}$ satisfy
$$
\dertn{k}h_j^{\pm}(t,\gm)\lesssim e^{-{\text \rm Re}(\frac{\gamma_jt^2}{2}+
\frac{\gb_jt}{\gm})}(1+|t|)^{a+k}\,\,\,{\text\rm (resp.)}.$$
Moreover, such estimates likewise hold for bases as in Proposition \ref{complex}
on their associated domains (\ref{sector}).
\end{prop}
\begin{proof}
We will first prove the estimates for $h_j=h_j^+,$ temporarily dropping the superscript.
In the case with domain $\BR^+\times(\gm_0,\infty)$
we note that for any sequence $\gm_l:$ $l=1,2,\hdots,$ tending to $+\infty$ there is a subsequence
(say $\gm_l$) so that each function $\dertn{k}\psi_j(\cdot,\gm_l)$ converges uniformly on compact sets of $\BR$
(ucs) to functions $\dertn{k}\zeta_j(t)$ respectively for each $0\leq k<n$ and where $\vec{\zeta}$ is a bases for ker$P_n(i\partial_t,t).$ Here $\vec{\zeta}$
satisfies
$W(\vec{\zeta})$ $=Ce^{\gamma t^2/2}$ for some constant $C$ (cf. \cite{w1,c}) where
$\gamma\df$ $\sum_{j=1}^n\gamma_j$.
Therefore, using Abel's formula along with the above estimate,
$W(\vec{\psi})(t,\gm)$ $\asymp$ $e^{\gamma t^2/2+a_n t/\gm}$
for sufficiently large $\gm_0$ and for $a_n$ as in (\ref{cE}).

$W_j(\vec{\psi})$ is a finite linear combination of $(n-1)$-fold products of the form
${\Pi}_{l\neq j}\partial_t^{\ga_l}\psi_l$
for distinct $l:$ $1\leq l\leq n$ and distinct $\ga_l:$ $0\leq \ga_l< n.$
Since $\sum\gb_j$ $=$ tr$\cD_1=$ tr$\cE_1$ $=a_n$
we find
$$W_j(\vec{\psi})(t,\gm)\lesssim e^{(\gamma -\gamma_j)t^2/2+(a_n-\gb_j)t/\gm}(1+|t|)^a$$
for some constant $a>0.$
%$h_j=A(\gm)W_j/W$
From \cite{w1} we recall that the functions ${\bar{h}_j}$ form a
basis for $\cL_{\gm}^*$ and, arguing by matching asymptotics, the
result of the proposition holds $\forall k$ in the present case.

The proof for the $h^-_j$'s follows in exactly the same way as above, using the corresponding estimates on
$\BR^-\times$ $(\gm_0,\infty)$ for some large $\gm_0>0.$
Finally,
the proof for bases defined as in Proposition \ref{complex} follows similarly, in fact more
readily since $\gm$ is restricted to compact subset of $\BC$, and we are done.
\end{proof}

%% file: ls5resub.tex
We start this section with
more definitions and notation. We will order the bases functions $\psi_j^{\pm}$ $\in$ ker $\cL_{\gm}$ (resp.)
according their asymptotic growth as indicated
by the pairs $(\gamma_j,\beta_{\gm,j})$.
Define $\Phi_{j}^{\pm}:$ $j=1,\hdots,n$
as an ordering of the $\Phi_j$'s as in $(\ref{gTh}),$
so that the following holds $\forall$
$1\leq j<$ $n$:
Re$\gamma_j\geq$ Re$\gamma_{j+1};$ Re$\gamma_j=$ Re$\gamma_{j+1}$ $\implies$ Re$\beta_j\gtrless$
Re$\beta_{j+1}$ (resp.) for all sufficiently large $\gm.$
Such ordered bases will be denoted in vector form
as $\vec{\psi}$ $\df$ $(\psi_1,\psi_2,\hdots,\psi_n)^{\dagger}$
with obvious superscript convention. Finally, since we will not
always keep precise track of power-function factors in our estimates, we introduce
the notation ${\lesssim}_{pol\vec{x}}$ (${\asymp}{_{pol\vec{x}}}$)
when the estimate $\lesssim$ (resp. $\asymp$) holds modulo factors of polynomial
growth in $\vec{x}:$ that is, the implied constants $C$ are each replaced by
$C(1+|\vec{x}|)^r$ for some sufficiently large, fixed $r>0$.

We will be using various basis transformations for our kernel spaces, but
for large $\gm>0$ we restrict our bases to those of a certain class developed from
canonical bases as in (\ref{mainest}).
We will
call a collection of bases
$\{\gpv_{j}^+(t,\gm)\}_{j=1}^{n}$ $,\{\gpv_{j}^-(t,\gm)\}_{j=1}^{n}$
{\it admissible} (or an {\it admissible pair}) if they
satisfy
$\vec{\gpv}^{\pm}=U^{\pm}\vec{\psi}^{\pm}$
(resp.) with $n\times n$ matrices $U^{\pm}$ satisfying the following: $U=U(\gm)$ has $\cC^{\go}$
entries on $(\gm_0,\infty);$
$[U]_{j,k}\lesssim$ $\gm^a$ $\forall j,k$ for some fixed $a>0;$ $[U]_{j,k}=0$
for $j<k$ (upper-triangular); and, $[U]_{j,j}$ $\gtrsim$ $\gm^{-b}$ $\forall j$ for some
fixed $b>0.$
In this case, it is easy to show
$$\partial_t^k\gpv_j^{\pm}(t){\asymp}{_{pol\gm}}e^{\Phi_j^{\pm}}(1+|t|^k):$$
$t\gtrless 0$ (resp.) for all sufficiently large $\gm>0$.
We further denote by $J_{\pm}$ the least index whereby
$j\geq J_{\pm}$ implies that
either Re$\gamma_j$ $<0$ or Re$\gamma_j$ $=0$
and Re$\beta_{\gm,j}^{\pm}$ $\gtrless 0$ for $(\gamma_j,\beta_{j}^{\pm})$
(resp.).
We simply distinguish the basis functions in decreasing order according
to their exponential growth for large $t$ in their respective domains.
We note that for a given admissible pair the
associated functions $\cH^{\pm}_j\df$ $W_j(\vec{\phi}^{\pm})/W(\vec{\phi^{\pm}})$
satisfy
$$\cH_j=h_j/[U]_{j,j}\lesssim_{pol\,(t,\gm)}e^{-\gamma_jt^2/2-\gb_jt/\gm}$$
on the corresponding domains.

To characterize the global behavior of our bases we introduce definitions regarding
transition matrices associated with admissible $\vec{\phi}^{\pm}$.
Given a permutation $\gs$ of
$\{1,2,\cdots,n\}$, denote by $I_{\gs}$ the $n\times n$ matrix
with elements defined by $[I_{\gs}]_{j,k}$ $=\gd_{j,\gs(j)},$ with
$\gd$ denoting the Kr\"oniker delta function. We characterize transition
(scattering) matrices $\vec{\phi}^+$ $=A\vec{\phi}^-$ as follows: Given an $n\times n$
matrix, $A,$ the expression $A\leftrightarrow I_{\gs}$ will mean
that $A=UI_{\gs}V$ for some invertible, upper-triangular $n\times n$ matrices
$U$ and $V.$ And, with slight abuse of notation $A\leftrightarrow
B$ will mean that $A\leftrightarrow I_{\gs}$ and $B\leftrightarrow
I_{\gs}$ both hold. We note that "$\leftrightarrow$" is an
equivalence relation on $GL(n,\BC)$. Given $J^{\pm}$ as above, we will say that
the permutation $\gs$ is {\it resolving} if $\gs(j)>J^-$ $\forall j$ $\leq$
$J^+$.

We will say the operator $\cL_{\gm}$ is
{\it regular} if on a real interval $(\gm_0,\infty)$
there is some admissible pair of bases $\vec{\gpv}^{\pm}$
and smooth functions $a_{j,l}^{\pm}(\gm)$ $\lesssim$ $\gm^r$
on $(\gm_0,\infty)$ so that $\forall j\leq J^{\pm}$
\begin{equation}\gpv_j^{\pm}(t,\gm)\notag
=\sum_{l> J^{\mp}}a_{j,l}^{\mp}(\gm)\gpv^{\mp}_l(t,\gm)\,\,{\text \rm (resp.)}
\end{equation}
for some fixed $r$ $>0$.

We now proceed with steps of constructing a parametrix under the hypothesis of
Lemma \ref{lem1} where we need only to consider the parameter $\gm$ in a (perhaps large) compact complex neighborhood of $0$: More precisely, we will suppose $\gm\in$ $\{re^{i\theta/2}|0\leq \theta\leq \pi\}$ for some fixed $r>0$. Let us restrict real $x,\xi$ so that $x$ is bounded, say $|x|<M,$ and $\xi>0;$
and, fix $0\leq\ga\leq \pi$ and $\gl>0$ as we set
\begin{equation}
t(x,\xi,\gm)=x\gm+(\rho e^{i\ga/2}+\xi)/\gm;\,\,z_0=x\gm+\gl e^{i\ga/2}/\gm.\label{tz}
\end{equation}
Given $\ga$ there is a positive $\gd$ and a sufficiently
large $R$ so that $z_0$ and $t$ take values in sets $\cK$ and $\fU^{\pm}_{\ga/2}$
as in (\ref{sector}), respectively, with (say)
$\gd=$ $\gu_1=$ $\gu_2$ and, hence, the estimates of Proposition \ref{complex} hold for some basis of ker$\cL_{\gm}$ for such values of $t,\gm$ and $z_0.$

We set out to choose a finite collection of such bases as follows:
We apply the Heine-Borel theorem to choose finitely many such intervals $\cU_{l}$
$=\{\theta: |\theta-\ga_l|<\gd_l\}:$ $l=1,\hdots,N$ (say)
with small positive $\gd_l$'s to form a
refined open cover of $[0,\pi/2].$
To each $l$ we may find corresponding constants $\gl_l$'s so large that
the associated $z_0,t$ of (\ref{tz}) lie in the respective domains
of those bases in Proposition \ref{complex}.

For each given $\cU_l^{\pm}$ let us reorder the corresponding pairs
$(\gamma_{l,j},\gb_{l,j}^{\pm})$ associated with bases $\vec{\psi}^{\pm}_l$
in the fashion as the $\vec{\psi}^{\pm}$ in Proposition \ref{complex}; and,
let us likewise define the indices
$J^{\pm}_l.$
Then for bounded $g$ $\in\cC^{\infty}(\BR)$ set
$z_l\df\gl_l e^{-i\ga_l},$ $\tilde{\xi}_{z_l}\df \xi +z_l$
and regard $t=t(x,\tilde{\xi}_{z_l},\gm)$ as in (\ref{tz}). Then we define for $[\vec{\psi}_l]_j$
$\df$ $\psi_{j;l}$
\begin{equation}\label{Fl}F_l^{\pm}(t,\gm)
\df\sum_{j=1}^n\psi_{j;l}^{\pm}(t,\gm)\int_{\cC_{l,j}^{\pm}(t,\gm)}\cH_{l,j}
(\zeta)g({\zeta}/{\gm}-\tilde{\xi}_{z_l}/{\gm^2})d\zeta
\end{equation}
with respective contours $\cC_{l,j}^{\pm}$ given by
$\cC_{l,j}^{\pm}(t,\gm)$
$=$ $\{s\gm+\tilde{\xi}_{z_l}/\gm|0\leq s\leq x,\pm\xi>0\}$ if $j>J^{\pm}_l$
and
$\cC_{l,j}^{\pm}(t,\gm)$ $=$  $\{s\gm+\tilde{\xi}_{z_l}/\gm|x\leq s<\infty,\pm\xi>0\}$ if
$j\leq J^{\pm}_l.$ Note that the contours are chosen so that ${\zeta}/{\gm}$$-\tilde{\xi}_{z_l}/{\gm^2}$
give only real values.

For each $l$ and choice of $\pm$ sign we have solutions
to $\cL_{\gm}f=g(x)$ for $(t,\gm)$ on
the corresponding domains. Since $\partial_x^jt$ $\lesssim 1$ $\forall j$, we find that for any given $\gk$
there is an $a>0$ so that, on their domains,
\begin{equation}\partial_x^kF_l^{\pm}\lesssim
(1+|t(x,\tilde{\xi}_{z_l},\gm)|)^a\lesssim(1+|\xi|)^a\notag\end{equation}
%on $\pm\xi>0$ (resp.)
is satisfied for $\forall k\leq \gk$ and $\forall l$. The proof
follows the analysis in \cite{w1} and details are deferred to the
Appendix (see Proposition \ref{dunno} and the comments that follow).
Define the $F_l^{\pm}$'s to be zero outside their corresponding
domains and now let $\chi_l(\theta)$ (for $\gth=\arg \gm$) be a
partition of unity subordinate to the $\cU_l^{\pm}$'s and let $c_l$
be quantities that are constant with respect to $x,\gm$ and $\xi$.
We then set
\begin{equation}\label{cG}
\cG(x,\xi,\gm)=\sum_{l=1}^mc_l\chi_l(\theta)(\gTh(\xi)F^+_l(t(x,\tilde{\xi}_{z_l},\gm))
+\gTh(-\xi)F^-_l(t(x,\tilde{\xi}_{z_l},\gm)))
\end{equation}
where $\gTh$ is the unit (Heaviside) step function to create a smooth function for $x\in[-M,M]$, $\xi\in\BR$,
$\gm\in$ $\{re^{i\gth}|0\leq\gth\leq\pi/2\}.$

We have not shown $\cL_{\gm}$ to have regular parametrices for complex $\gm$, as per Definition \ref{def1}, but
we can establish partial construction of solutions to $Lu=f$ with, as yet, no additional conditions.
We state
\begin{prop}\label{partsol}
Given $g\in\cC^{\infty}(\cI)$ for a neighborhood of $0\ni$ $\cI$ $\subset\BR,$ there is
a function $f(x,y,w,\xi,\gth)$ $\in$ $\cC^{\infty}(\cI\times\BR^3\times[0,\pi/2])$ satisfying
$$P(X,Y)f=g(x)e^{-i(y\xi+w\gm^2)}\sum_{l=1}^m\chi_l(\gth)e^{-iyz_l}$$ with
$\gm$ $=$ $\gl e^{i\gth}$ for fixed $\gl>0$.
\end{prop}
\begin{proof}
We employ (\ref{cG}) setting
$c_l=$ $e^{-iz_ly}$ for each $l$
and we set $f$ $=$ $(i/\gm)^{n}e^{-i(y\xi+w\gm^2)}\cG$.
%$\cF(x,y,w,\gm,\xi)$ $\df$ $$\left(\frac{i}{\gm}\right)^n
%\sum_{l=1}^m\chi_l(\theta)(\gTh(\xi)F^+_l(x,\xi+z_l,\gm)+\gTh(-\xi)F^-_l(x,\xi+z_l,\gm))$$
%\begin{align}
We note that, on the respective domains $\cU^{\pm}_l,$ the $F_l^{\pm}$'s
%(dependence on $x,\xi,\gm$ understood
as in (\ref{Fl}) now satisfy
\begin{align}L(e^{-i(y\xi+\gm^2w)}F)&=e^{-i(y(\xi+z_l)+w\gm^2)} P(\partial_x,-i(\xi+z_l+\gm^2x))F\notag\\
&=e^{-i(y(\xi+z_l)+w\gm^2)}\left({-i}{\gm}\right)^n\cL_{\gm}F_l\notag\\
&=e^{-i(y(\xi+z_l)+w\gm^2)}g(x)\notag
\end{align}
So, we obtain our desired result as we compute
\begin{align}
Lf&=\sum_{l=1}^m\chi_l(\gth)e^{-i(y(\xi+z_l)+w\gm^2)}(\gTh(\xi)\cL_{\gm}F^++\gTh(-\xi)\cL_{\gm}F^-)\notag\\
&=\sum_{l=1}^m\chi_l(\gth)e^{-i(y(\xi+z_l)+w\gm^2)}(\gTh(\xi)+\gTh(-\xi))g(x)\notag\\
&=\sum_{l=1}^m\chi_l(\gth)e^{-i(y(\xi+z_l)+w\gm^2)}g(x)\notag\end{align}
\end{proof}
We now present the

\noindent${Proof}\,of\,Lemma \,\ref{lem1}$:
We replace $g(x)$ by functions $\hat{g}(x,\xi,\eta)$ for
$g$ $\in$ $\cC_0(\BR^3),$ involving the other Fourier variable $\eta$
as we substitute
$\gm=\sqrt{|\eta|}$ into (\ref{tz}).
First we show that for $F$ as in Definition \ref{def1}
$$\cF(x,\xi,\eta) \df F(t(x,\xi,\sqrt{|\eta|}),\sqrt{|\eta|})$$
defines a function in $\cC^{\infty}(\BR)\times$ $\cS(\BR)\times$ $\cS((\gm_0^2,\infty))$:
By linearity of the operator $\cL_{\gm}$ we have that for any $r>0$ and index $m,$
\begin{align}\partial_{x}^k\cF \lesssim &
|\eta|^{-n/2}(1+|\sqrt{\eta}|)^m(1+|\xi|+|\eta|)^{-r}(1+|x|\sqrt{|\eta|}+|\xi/\eta|)^a\notag
\end{align}
for $|x|<M,$ $|\eta|>\gm_0^2,$ $\xi\in\BR$ $\forall k\leq m$.
So $\forall$ $\ga>0,$ $\partial_{x}^k\cF$  $\lesssim$ $(1+|\xi|+|\eta|)^{-\ga}$
by setting $r>\ga+a+m/2.$

Letting $\phi_l(\gth)$ denote the characteristic function of the associated
set $\cU_l$, we replace $g$ in Proposition \ref{partsol} by
$\sum_{l=1}^m\phi_l(\gth) \hat{g}(x,\xi+z_l,\pm\gm^2)$ to redefine $F_l^{\pm}$ (resp.)
and set $\gl=\gm_0^2.$
Likewise, the various partial derivatives $\partial_x^kf(x,y,w,\xi,\gth)$
are each majorized by $(1+|\xi|)^{-s}$ for
any $s>0.$
We apply $L$ to $F=f_1+f_2$ where
$$f_1\df
%\frac{1}{2\pi}\int_{\BR^2}
\frac{1}{2\pi}\int_{\BR}\int_{|\eta|>\gm_0^2}
(i/\sqrt{|\eta|})^n
e^{-i(y\xi+w \eta)}(\cF^{+}(x,\xi,\eta)+\cF^{-}(x,\xi,\eta))d\eta d\xi$$
$$Lf_1=\frac{1}{2\pi}\int_{\BR}\int_{|\eta|>\gm_0^2}e^{-i(y\xi+w \eta)}\hat{g}(x,\xi,\eta)d\eta d\xi$$
$$f_2\df\frac{1}{2\pi}\int_{\BR}\int_{0}^{\pi/2}f(x,y,w,\xi,\gth)\gl^2ie^{2i\gth}d\gth d\xi$$
$$Lf_2=\frac{1}{2\pi}\int_{\BR}\int_{\cC}
\hat{g}(x,\xi+z_l,\zeta)e^{-i(y\xi+w \zeta)}
\sum_{l=1}^m\chi_l\left(\frac{\arg \zeta}{2}\right)e^{-iyz_l}d\zeta d\xi
$$
where $\zeta=\gm^2$ and $\cC$ is the complex contour given by
boundary of the upper half of the disc centered at $0$ of radius $\gm_0^2.$
By the analyticity and integrability of $\hat{g}$ in the variable $\xi$
we may apply Cauchy's integral theorem to obtain
$$\int_{\BR}e^{-iy(\xi+z_l)}\hat{g}(x,\xi+z_l,\zeta)d\xi
=\int_{\BR}e^{-iy\xi}\hat{g}(x,\xi,\zeta)d\xi
$$ for $(\arg\zeta)/2$ $\in$ supp$\chi_l$. We now apply the Fubini-Tonelli
theorem along with Cauchy's integral theorem in the variable $\zeta$ to obtain
$$Lf_2=\frac{1}{2\pi}\int_{\BR}\int_{\cC}\sum_{l=1}^m\chi_l\left(\frac{\arg \zeta}{2}\right)\hat{g}(x,\xi,\zeta)e^{-i(y\xi+w \zeta)}d\zeta d\xi$$
$$=\frac{1}{2\pi}\int_{\BR}\int_{\cC}\hat{g}(x,\xi,\zeta)e^{-i(y\xi+w \zeta)}d\zeta d\xi$$
$$=\frac{1}{2\pi}\int_{\BR}\int_{-\gm_0^2}^{\gm_0^2}\hat{g}(x,\xi,\eta)e^{-i(y\xi+w\eta)}d\zeta d\eta$$
$$LF=
%$$\frac{1}{2\pi}\int_{\BR}\left[\int_{|\eta|>\gm_0^2}e^{-i(y\xi+w \eta)}\hat{g}(x,\xi,\eta)d\eta
%+\frac{1}{2\pi}\int_{-\gm_0^2}^{\gm_0^2}e^{-i(y\xi+w \eta)}\hat{g}(x,\xi,\eta)d\eta\right] d\xi$$
\frac{1}{2\pi}\int_{\BR^2}e^{-i(y\xi+w \eta)}\hat{g}(x,\xi,\eta)d\eta d\xi=(\hat{g}\check{)}(x,y,w)=g(x,y,w)$$
\qed

We are now prepared to state a main result from which we may determine solvability
through the representations $\gL^{\pm}_{\gm}$ for large $\gm:$
\begin{thm}\label{reg}
The operator $L$ as in (\ref{op}) is locally solvable if both $\cL^{\pm}_{\gm}$ are regular.
\end{thm}
\begin{proof}
Since our proof repeats content of previous works, we will give a sketch of proof
and defer details to these works.
We will show that an operator $\cL_{\gm}$ has a regular parametrix if
it is regular:
\begin{equation}\label{F}
F(t,\gm)=\sum_{l=1}^{J^+}\ga_l(\gm)\gpv_l^+(t,\gm)+\sum_{l=1}^{J^-}\ga_l(\gm)\gpv_l^-(t,\gm)+K(t,\gm)
\end{equation}
where $K(t,\gm)=$
$$\sum_{l=0}^n\gpv_l^+(t,\gm)\int_0^t\cH^+_l(\gt,\gm)g(\gt)d\gt =\sum_{l=0}^n\gpv_l^-(t,\gm)\int_0^t\cH^-_l(\gt,\gm)g(\gt)d\gt
$$
(cf. (12) and (13) of \cite{w3}).
We set
$\ga_l^{\pm}(\gm)\df\int_{\pm\infty}^0\cH_l^{\pm}(\gt,\gm)g(\gt)d\gt$ (resp.)
so that we can write
$$
F(t,\gm)=\sum_{l=1}^{J^+}\phi_l(t,\gm)\int_{+\infty}^t\cH_l(\gt,\gm)g(\gt)d\gt+\sum_{l=1}^{J^-}\ga_l(\gm)\gpv_l^-(t,\gm)
$$
$$
F(t,\gm)=\sum_{l=1}^{J^-}\phi_l(t,\gm)\int_{-\infty}^t\cH_l^-(\gt,\gm)g(\gt)d\gt+\sum_{l=1}^{J^+}\ga_l(\gm)\gpv_l^+(t,\gm)
$$

The estimates on $K$ and its derivatives $\partdertn{k}K$
follows as in Section 1 of \cite{w1}, applying the Chain Rule along with
integral estimates as in Proposition \ref{dunno}, to obtain
$\partdertn{k}F(t,\gm)\lesssim_{pol\,(t,\gm)}1$
on $\BR\times (\gm_0,\infty)$ for each $k$.
\end{proof}

%% file: ls6resub.tex
We see that when Re$\gamma_j$ $=0$ the corresponding bases function $\psi_j(t,\gm)$, for finite $\gm$, may have asymptotic growth far different than
that of a corresponding limit $\zeta_j(t)$. For instance,
 a limit function $\zeta_j$ (say) may have polynomial growth where every sequence $\psi(\cdot,\gm_l):$ $l\rightarrow \infty$ tending to it may have exponential growth or decay in the variable $t$. We distinguish
a particular subclass of operators where this discrepancy does not take place. We introduce
\begin{definition} We will say that the polynomial $P$ has property
$\cG$ if $P_n$ is generic and the $\gamma_j,$ $\gb_j:$ $1$ $\leq j$ $\leq n$ associated with $P$ satisfy
$${\text \rm Re}\gamma_j =0 \implies{\text \rm Re}\gb_j =0.$$
\end{definition}
We note that this property depends only on the coefficients of $P_n$ and $P_{n-1}.$
We are ready to state \begin{thm}\label{solve}
Suppose that $L=P(X,Y)$ where $P$ has property $\cG$.
Then, the operator $L$ is locally solvable
if ker$(\cL_{\infty}^{\pm})^*$$\bigcap$ $\cS(\BR)$ contains only the zero function for each choice
of $\pm$ sign.
\end{thm}
\begin{proof}
We will show that an operator $\cL_{\gm}$ is regularizable when $\cL_{\infty}^*$
satisfies the hypothesis.
Here, $J^+=J^-\df J$ and our result follows as in proof of Corollary 4.3 of \cite{w3}
since the associated transition matrix $A$ has real analytic coefficients and tends
to a finite limit as $\gm\rightarrow$ $\infty.$
We deduce that there is an admissible pair of bases $\vec{\phi}^{\pm}$ for which
$\vec{\phi}^{\pm}=$ $\gL^{\pm}I_{\gs}\vec{\phi^{\mp}}$ (resp.) for square matrices $\gL^{\pm}$ and $I_{\gs}$
satisfying the following for each choice of $\pm$ sign and sufficiently large $\gm>0$:
$\gL$ is lower triangular with ones on the main diagonal
where $\gL(\gm)$ $\rightarrow I$ as $\gm\rightarrow$ $\infty;$
and, $[I]_{j,k}=$ $\gd_{j,\gs(j)}$ where $\gd$ is the Kr\"oniker delta function and $\gs$ is a resolving permutation.
By carrying out the matrix multiplication we find that $\cL_{\gm}$ is regularizable.

Since both $\cL^{\pm}_{\gm}$ are regularizable, the proof is complete
by applying Theorem \ref{reg}.
\end{proof}
From \cite{w1} the hypotheses on $\cL_{\infty}^{\pm}$ are equivalent to the local solvability
of $P_n(X,Y)$, whereby we immediately conclude the following:
\begin{cor}\label{cora} Suppose that $L=P(X,Y)$ is locally solvable where
$P$ is a generic polynomial of order $n\geq 2$. Then,
the operator $L+K$ is locally solvable for any $K$ contained in the subalgebra of $\fh_1^{\BC}$
generated by $X$ and $Y$
of order less than or equal to $n-2$.
\end{cor}
\begin{cor}
Let $L$ and $K$ be as in Corollary \ref{cora} except that the characteristic roots of $P$ each have
non-zero real parts. Then the operator $L+K$ is locally solvable for any such $K$
of order less than or equal to $n-1$.
\end{cor}

We demonstrate that regularity is not a precise condition for local solvability of our operators
as an even weaker condition on $\cL_{\gm}^{\pm}$ still assures solvability. We introduce
\begin{definition}
An operator $\cL_{\gm}$ will be called quasi-regular if
there is a finite, refined open cover $\cI_l:$ $l=1,2,\hdots,m$
of a semi-infinite interval $[\gm_0,\infty)$ along with
admissible pairs $\vec{\phi}^{\pm}_l,$ with components
$[\vec{\phi}^{\pm}_l]_{j}$ $\df$ $\phi_{j;l}^{\pm}$ (resp.)
that satisfy estimates as in (\ref{reg}) but with $\gm$ restricted to $\cI_l$.
\end{definition} We note that the $\cI_l$'s can each be assigned
a unique resolving permutation $\gs_l$ associated with the transition matrix
$A_l$ defined on $\cI_l.$  We note further that it is trivial to show that
a regularizable operator is also
quasi-regular.  We will see (in Section \ref{s6}) that
for a large subclass
of our operators the aforementioned condition is exact.

\begin{thm}
An operator $L=P(X,Y)$ is locally solvable if the operators $\cL_{\gm}^{\pm}$
are both quasi-regular.
\end{thm}
\begin{proof}
Via a change of variables, if necessary, it will
suffice to show that an operator $\cL_{\gm}$ has a regular parametrix if it is quasi-regular.
We can construct functions $F_l(t,\gm)$ as in (\ref{F}) but with $F$ restricted to $\BR\times\cI_l.$
Then for a smooth partition of unity $g_l(\gm)$ subordinate to the associated $\cI_l:$ $l=1,\hdots,m$
we set
$\cF(t,\gm)$ $\df$ $\sum_{l=1}^mg_l(\gm)F_l(t,\gm).$ Local solvability of $L$ then follows as in the proof of Theorem \ref{reg}
and the cited references therein.\end{proof}

%% file: ls6bresub.tex
We will develop criteria for local non-solvability of $L$ in terms
of representations $\gL_{\gm}^{\pm}$ - particularly through the
adjoints of $\cL_{\gm}^{\pm}.$  We first need to verify that our
approach to solving (\ref{transform}) has an analogue as applied to
adjoints of our partial differential operators. For constructing
functions in ker$L^*$ the method is clear via
\begin{prop}
Under $\gL_{\gm}^{\pm}$, the operator $L^*$ has representations given by
$\gL^{\pm}_{\gm}(L^*)$ $=(i\gm)^n(\cL^{\pm}_{\gm})^*$ (resp.).
\end{prop}
\begin{proof}
We will drop $\pm$ the superscript. It is not difficult to show that the conclusion of our proposition
holds for operators homogenous in $X$ and $Y$ (cf. Proposition 2.3 \cite{w1}) so that
$\gL_{\gm}(P_l(X,Y)^*)=(\gL_{\gm}(P_l(X,Y)))^{*}$ $\forall l$.
It follows that $\gL_{\gm}(L^*)$ $=(\gL_{\gm}(L))^*$ and our result is immediate.
\end{proof}
Our next result follows as in Section 2 \cite{w1} and, hence, we provide
here only a sketch of proof of the following
\begin{thm}\label{nonsol1}
Suppose that $\exists\psi\in$ ker$\cL_{\gm}^*$ such that $\psi\cdot\chi$
$\in$ $\cS(\BR\times \cI)$$\setminus$$\{0\}$
for $\chi$ $\in$ $\cC_0^{\infty}(\cI)$ where $\cI$ is a bounded open subinterval of $\BR^+$.
Then, the associated operator $L$ is not locally solvable.
\end{thm}
\begin{proof}
We suppose (perhaps after a change of variables) that $\cI$ $=(1,2).$
It follows as in Proposition 2.4 \cite{w1} that
there is a non-trivial function $\Psi$ $\in$ ker$L^*$ $\bigcap$ $\cS(\BR^3)$
with $\gm$ $=$
$\sqrt{\xi_3}$ given by:
$$\hat{\Psi}(x,{\xi_2},{\xi_3})=\psi(t(x,\xi_2,\gm),\gm)F(\xi_2,\gm),$$
where $F\in\cC_0(\BR^2)$ and
supp$F(\xi_2,\gm)$ $\subseteq$ $\{ \vec{\xi}| \, |\xi_2|<1 \,\,\&\,\,\xi_3\geq 1\}$.

We set $v_{\gt}\df\phi(\vec{x})\Psi(\gt x,\gt y, \gt^2 w)$ and $F_{\gt}$ $\df$ $\phi(x/\gt,y/\gt,w/\gt^2)\Psi(\vec{x})$
for some $\phi\in$ $\cC_0(\BR^3)$ such that $\phi=1$ on a neighborhood of the origin. These functions
are constructed likewise to Propositions 2.5 and 2.7 of \cite{w1}, so that in the Sobolev norm $v_{\gt}$ satisfies
$||L^*v_{\gt}||_{(\gn)}$ $\lesssim \gt^{-a}$ for all
$\gt>1$ for any fixed $\gn$ and positive $a.$

To complete our result, it suffices to show that for any integer $N>0$ there are constants $C_1,C_2>0$
so that $||v_{\gt}||_{(-N)}>C_1\gt^{-N-2}$ and $||v_{\gt}||_{(-N-3)}<C_2\gt^{-N-4},$
thereby showing that for sufficiently large $\gt>1$
H\"ormander's criteria from Lemma 26.4.5 \cite{ho2} is violated (see also
(2.19) \cite{w1}). The desired estimates
follow in the same manner as (2.13) through (2.18) \cite{w1}, considering Remark \ref{rem1} (Appendix).
\end{proof}

We apply the above results to develop a priori criteria for non-solvability.
Given (distinct) characteristic roots $\{\gamma_j\}_{j=1}^n$ and associated $\gb_j$ as in (\ref{beta}), let us denote $\cB^{\pm}\df\{j|\pm\text{\rm Re}\gamma_j>0\},$
and $\cE^{\pm}\df\{j|\text{\rm Re}\gamma_j=0\, \&\, \pm\text{\rm Re}\gb_j> 0\}$
(resp.).
We state our result in terms of the
cardinality ($card$)
of these sets:
\begin{thm}\label{nonsol2}
$L=P(X,Y)$ is not locally solvable if either of the following holds:
\begin{itemize}\item[1)] $card(\cB^+\bigcup\cE^+)$ $+$ $card(\cB^+\bigcup\cE^-)>n;$ or,
\item[2)]$card(\cB^-\bigcup\cE^+)$ $+$ $card(\cB^-\bigcup\cE^-) >n$.
\end{itemize}
\end{thm}
 \begin{proof}
By choosing the appropriate representation, we may suppose that case 1) holds: By Remark \ref{remark1} along with (\ref{D1}) we see that
$\cE^+\bigcup\cE^-$ is the same under each such representation.
We note further that $\phi(t)\in$ ker$P(i\partial_t,\gm t)^*$ $\iff$ $\phi(-t)$
$\in$ ker$P(-i\partial_t,-t)^*$ where the associated parameters $\gamma_j,$ $\beta_j$
transform as $(\gamma_j,\beta_j)$ $\rightarrow$ $(\gamma_j,-\beta_j)$
for each $j.$

By dimensional arguments it follows that for any admissible pair of bases $\vec{\phi}^{\pm}$ of ker$\cL_{\gm}$ the
associated transition matrix $A$ satisfies $A\leftrightarrow I_{\gs}$
for a non-resolving $\gs$ for $\gm$ on a non-empty open interval.  Then there is a pair of bases
$\vec{\fh}^{\pm}$ of ker$\cL_{\gm}^*$
given by
$\vec{\fh}^-=(A^{-1})^{\dagger}\vec{\fh}^+$
(admissible after reordering $h_j\rightarrow $ $h_{n+1-j}$ $\forall j$)
where $(A^{-1})^{\dagger}\leftrightarrow I_{\gs}.$
It follows as in Corollary 4.3 \cite{w3} that
$\exists l$ $\in$ $\cB^+$ $\bigcup$ $\cE^+$ so that
\begin{equation}\label{adj}
\fh_l^-(\cdot,\gm)=\sum_{j\leq J^+}a_j(\gm)\fh_j^+(\cdot,\gm)
\df \overline{\psi}(\cdot,\gm)\end{equation}
where the $a_j$ are real analytic functions, not all trivial.
By Proposition \ref{moremainests} and analyticity arguments on the $a_j$'s,
there is a bounded, non-empty interval $\cI$ $\subset R^+$ so that $\forall k$
$\dertn{k}\psi(t,\gm)$ $\lesssim e^{-\gd_k|t|}$ holds
on $\BR\times \cI$ for some $\gd_k>0$.
 The result now follows by applying Proposition \ref{nonsol1}.
\end{proof}

\noindent{\it Proof of Theorem \ref{thm2}:} Let $\tilde{\gm}>0$ be the accumulation
point. By analyticity of transition matrices $A(\gm)$
we find that
$A\leftrightarrow I_{\gs}$
for all $\gm\in\cI$ containing $\tilde{\gm}$ for
some fixed, non-resolving $\gs.$ It then follows that a function
$\psi(\cdot,\gm)$ can be constructed as done in (\ref{adj}).
\qed

%% file: ls7resub.tex
%What if $\cL_{\gm}$ was not regularizable?

%Let up denote by $h_{j;l}$ those functions as in Proposition \ref{thehs}
%constructed in the same fashion by bases $\vec{\phi}_l^{\pm}$ where $\gm$ is restricted
%to $cI_l$.

\begin{prop}\label{pcwsreg}
If $\cL_{\gm}$ is not quasi-regular, then there is a sequence $\gm_l:l=1,2,\hdots$
with $\gm_l$ $\rightarrow$ $+\infty$
and functions $Y_{j}^{\pm}(t,\gm)$ $\asymp_{pol\gm}$ $\bar{h}_{j}(t,\gm)$ for $h^{\pm}_j$'s as in Proposition \ref{thehs}
for which the following hold:  $\exists k\leq J^-$ and coefficients $\ga_j(l)$ so that
$$\cF_l(t)\df Y_{k}^-(t,\gm_l)=\sum_{j=1}^n\ga_j(l)Y_{j}^+(t,\gm_l)$$
where $\forall j$
$\ga_j(l)$ $\leq$ $C\gm_l^a$ for some fixed $C,a>0$ and where
for any $B>0$ we find
$\ga_j(l)< $ $C_B\gm_l^{-B}$
is satisfied for some constant $C_B>0$ for each $j>J^+$. Moreover, the $\cF_l$'s
may be chosen so as to converge
(ucs) to a
non-trivial function in ker $\cL_{\infty}^*$.
\end{prop} The proof follows as that of Proposition 3.2 $\cite{w3}$ and further elaboration is deferred to
the Appendix (see Remark \ref{w3stuff}).

\begin{thm}
Suppose that for an operator $L=$ $P(X,Y)$ one of operators $\cL^{\pm}$ (say $\cL$)
is not quasi-regular. Then $L$ is not locally solvable if either of the following additional conditions hold:
\begin{itemize}\item[1)]
$P$ has property $\cG$;
\item[2)]
%ker$\cL^{*}_{\infty}$$\bigcap$$\cS(\BR)$$\neq$$\{0\}$ and
Functions $\cF_l$ $\in$ ker$\cL_{\gm_l}^*$ as in Proposition \ref{pcwsreg} can be chosen to satisfy $\cF_l(t)$
$\lesssim_{pol\,\,t} e^{bt}$ for some $b>0$ with implied bounds uniform in $l$.\end{itemize}
\end{thm}
\begin{proof}
We may, perhaps after a change of variables, suppose that it is
the operator $\cL_{\gm}=\cL_{\gm}^+$ that is not
piecewise regularizable. We will show in this case that necessary criteria, in the form of an inequality involving $L^*,$ will be violated.
From \cite{ho2} we find that if $L$ is
locally solvable near $0,$ then the following holds:
$\forall\epsilon$ $>0,$ $\exists N>0$ such that
\begin{equation}\label{horm} |\int \phi\bar{\Psi}|\leq
N||\phi||_{\cC^{N}}||L^*\Psi||_{\cC^{N}}
\end{equation}
for every $\phi,$ $\Psi$ $\in \cC^{\infty}(\BR^3)$ supported in
$|(x,y,w)|<\epsilon.$

We begin our construction of functions $\phi_l,\Psi_l:$ $l=1,2,\hdots$ which for any given $N$ violates this
inequality for sufficiently large $l.$ Under our hypothesis we conclude from Proposition \ref{pcwsreg} that
there is an increasing sequence $1<$ $\gm_l$ $\rightarrow$ $+\infty$ and non-trivial functions $\cF_l$ $\in$ ker$\cL^*_{\gm_l}$ $\forall l$ satisfying the following:
There are real constants $M_1<M_2$ and a positive $\gd>0$ so that $\cF_l(t)$ $>1$ $\forall l$ on the interval
$(M_1-\gd,M_2+\gd);$
and, for any given $k$
there is an exponent $a\geq 0$ so that for some constants $r,s>0$
\begin{equation}\label{cF}\dertn{j}\cF_l(t)\lesssim (1+|t|)^a
(\fC(l)e^{s Q(t)}+e^{-r Q(t)})
\end{equation}
where $Q(t)$ $\df t^2$ and where, for any $\ga>0,$ $\fC(l)\lesssim \gm_l^{-\ga}$ holds uniformly $\forall l$.
Letting $\gn_l$ $\df\sqrt{\ln\gm_l},$ we
choose a non-negative $h$ $\in$ $\cC_0^{\infty}(\BR^3)$
supported in $|x|\leq\gl_2,$ $|y|\&|w|\leq\gep/2$ so that
$h(\vec{x})=1$ on $|x|\leq\gl_1,$ $|y|\&|w|<\gep/3;$
here, $1\leq \gl_1<\gl_2$ are constants yet to be determined.
We set $h_l(\vec{x})$ $\df$ $h(x\gm_l/\gn_l,y,w)$ (noting that
$ \gm_l/\gn_l$ tends to $\infty$ as $l$ tends to $\infty$).
Finally, we choose a non-negative $\chi$ $\in$ $\cC_0^{\infty}(\BR)$ supported
in $(M_1,M_2)$ so that $\int_{\BR}\chi(s) ds$ $=1.$

We now introduce functions
$$\Psi_l(\vec{x}) \df h_l(\vec{x})\int_{\BR}e^{-i(\gm_l^2w+\xi y)}\gm_l^{-1}\chi(\xi/\gm_l)\cF_l(x\gm_l+\xi/\gm_l)d\xi\df h_l(\vec{x})f_l(\vec{x})$$
We see that
$f_l(\vec{x})\df$ $\int_{M_1}^{M_2}e^{-i(\gm_l^2w+\gm_lu y)}\chi(u)\cF_l(x\gm_l+u)du$ $\in$ $\cC^{\infty}(\BR^3).$
Choose a non-negative $\phi(\vec{x})$ $\in\cC_0^{\infty}(\BR)$ supported in $|\vec{x}|<\gep/2$ so that
$\phi(\vec{x})$ $=1$ on $|\vec{x}|<\gep/3$ and set $\phi_l$ $\df$ $\phi(x,y\gm_l,w\gm_l^2).$

We now compute an upper bound of the RHS of (\ref{horm}), first by estimating the norms of $L^*\Psi_l.$
We compute
$L^*\Psi_l$ $=\cP_l f_l$ where $\cP_l$ $\df [L^*,h_l]$ is a partial differential operator of order $n$ supported in
$$U_l\df \{\vec{x}|\gn_l\gl_1\leq|x|\gm_l\leq\gl_2\gn_l\, \&\, |y|,|w|\leq \gep/2\}.$$
By the Leibniz Rule we may write $\cP_l=\sum_{|\ga|=1}^na_{\ga;l}(\vec{x})\partial_{\vec{x}}^{\ga}$ where the coefficients satisfy $a_{\ga;l}$ $\lesssim$ $\gm_l^{n-|\ga|}$ uniformly on $U_l.$
The operator
$\partial_{\vec{x}}^{\gb}\cP_l,$ of order $n+|\gb|$, behaves
likewise, but $\forall \ga$ the coefficient multiplying $\partial^{\ga}_{\vec{x}}$ is dominated by
$\gm_l^{n+|\gb|-|\ga|}.$
Further, for $\vec{x}\in$ $U_l,$ $Q(\gl_1\gn_l+M_1)$ $\leq$ $Q(x\gm_l+\xi/\gm_l)$ $\leq Q(\gl_2\gn_l+M_2)$.
For any multi-index $|\ga|$ $\leq k$ we have the following bounds (uniform also in $l$)
\begin{equation}\label{cPest}\partial^{\ga}_{\vec{x}}\cP_l \circ f_l(\vec{x})
\lesssim \gm_l^{2k+n}(1+\gn_l)^{a+n+k}(\fR_l+\fS_l)
\end{equation}
where we have set $\fR_l$ $\df$ $\fC(l)e^{s\gl_2^2\gn_l^2}$ $=\fC(l)\gm_l^{\gl_2^2s}$
and $\fS_l$ $\df$ $e^{-r\gl_1^2\gn_l^2}$ $=\gm_l^{-\gl_1^2r}.$
So by setting $k=N,$ $ \gl_1^2 > (3N+2n+a)/r+B$ and fixing $\gl_2$ $>\gl_1,$ we obtain
$||L^*\Psi||_{\cC^N}\lesssim (\gm_l^{-B}+\fC(l)\gm_l^{\gl_2^2s}).$
Then, for any integers $N,B\geq 0,$
$\exists C_{_{N,B}}>0$ so that $||L^*\Psi||_{\cC^N} \leq \cC_{_{N,B}}\gm_l^{-B}.$

We now estimate the norms of $\phi_l.$ For any $k$
the derivatives $\partial^{\ga}_{\vec{x}}\phi:$ $|\ga|\leq k$
are compactly supported in $|\vec{x}|<\gep.$ Then, by the Chain Rule,
$\partial_{\vec{x}}^{\ga}\phi_l(\vec{x})$ $\lesssim \gm_l^{2k}$ for such $\ga.$
Furthermore, we find that the RHS of (\ref{horm}) is majorized by $\gm_l^{2N-B}$ uniformly in $l.$

We now find a lower bound on LHS of (\ref{horm}). Let us restrict $0\leq \gep\leq \pi/4$ (say).
Then, for $\vec{x}$ $\in$ supp$\phi_l,$ Re$e^{i(\gm^2_ly+\gm_lw)}$ $\geq $ $\sqrt{2}/2.$
And, for sufficiently large $l$ (st. $\gn_l>\gl_1/\gd$) and $\vec{x}$ $\in$ supp $h_l,$
we have $\cF(x\gm_l+\xi/\gm_l)\geq 1$. Hence,
$${\text \rm Re}(\phi_l(\vec{x})\bar{\Psi}_l(\vec{x}))
\geq \frac{\sqrt{2}}{2}\phi_l(\vec{x})h_l(\vec{x})\int_{\BR}\chi(u)du; $$
and, for sufficiently large $l$ (as above and for $\gm_l/\gn_l>6\gl_1/\gep$)
$$\int_{\BR}h_l(\vec{x})\phi_l(\vec{x})d\vec{x}\geq
\int_{-\frac{\gl_1\gn_l}{2\gm_l}}^{+\frac{\gl_1\gn_l}{2\gm_l}}
\int_{-\frac{1}{2\gm_l^{2}}}^{+\frac{1}{2\gm_l^{2}}}\int_{-\frac{1}{2\gm_l}}^{+\frac{1}{2\gm_l}}1\,dydwdx
=\frac{\gl_1\sqrt{\ln\gm_l}}{\gm_l^4}$$
Therefore LHS of (\ref{horm}), for sufficiently large $l$,
is bounded below by $\lambda_1\gm_l^{-4}/\sqrt{2}.$ The condition on $L$ can
therefore be shown not to hold as we choose $\gl_1>0$ so large that we may fix
$B>4+2N$. The condition
(\ref{horm}) is then violated for all sufficiently small $\gep>0$ for any index $N$ for all sufficiently large
$l$; and, hence, the result of case 1 is shown.

The proof of case 2 is similar to that of case 1 by replacing $Q$ in (\ref{cF})
with $Q_l(t)\df$ $|t|/\gm_l$ and by redefining
$\gn_l\df$ $\gm_l\ln\gm_l,$ $h_l(\vec{x})$ $\df$ $h(x\gn_l/\gm_l,y,w).$
Then (\ref{cPest}) holds for $\fR_l$ $\df$ $\fC(l)e^{b\gl_2\gn_l/\gm_l}$ $=\fC(l)\gm_l^{b\gl_2}$ and
$\fS_l$ $\df$ $e^{-r\gl_1\gn_l/\gm_l}$ $=\gm_l^{-r\gl_1}$ for $l=1,2,\hdots$ 
We may choose $\gl_1<\gl_2$ so large that RHS of (\ref{horm}) is likewise majorized by
$\gm_l^{-B}$ for any desired $B>0$ as $l\rightarrow \infty.$ In this case we compute
$\sqrt{2}\text{Re}\int_{\BR}\phi\bar{\Psi}d\vec{x}$ $\geq$ $\gl_1\gm_l^{-3}/\ln\gm_l$ for sufficiently large $l$, so
that LHS has a lower bound the same as that of case 1. The criteria (\ref{horm}) are thus violated
and the proof is complete.
\end{proof}
We find that quasi-regularity of both $\cL^{\pm}_{\gm}$ is an exact condition for various subclasses of our operators.
\begin{cor}\label{quasi}
If $L=P(X,Y)$ where $P$ has property $\cG,$
then $L$ is locally solvable if and only if both $\cL^{\pm}_{\gm}$ are quasi-regular.
\end{cor}
\begin{cor}
Suppose that the characteristic roots of (generic) $P_n$ associated with $L=P(X,Y)$
satisfy either Re$\gamma_j$ $\neq$ $0$ $\forall j$ or Re$\gamma_j$ $=$ $0$ $\forall j$.
Then $L$ is locally solvable if and only if both $\cL^{\pm}_{\gm}$ are quasi-regular.
\end{cor}
At this point one may suspect that local solvability of $P(X,Y)$ is not always assured 
whenever $P_n(X,Y)$ is solvable. Indeed, we give
examples of non-solvability in such some such cases in the next section.

%% file: ls8resub.tex
We now give some examples of solvable and non-solvable operators in the second-order case.
We start with criteria for local non-solvability for the adjoints of the operators
\begin{equation}\label{secordex}
 L=-X^2-i{a}_1YX+{a}_2Y^2-i{\ga}[X,Y]-ib_1X+b_2Y+c
\end{equation}
where $a_k,b_k,\ga,c$ complex numbers with $a_1^2\neq 4a_2$.
We compute
 $$\cL_{\gm} =\dertn{2}+a_1t\dert+a_2t^2+\ga +\frac{b_1}{\gm}\dert+\frac{b_2}{\gm}t+\frac{c}{\gm^2}$$
 With $\gamma_j$ as the characteristic roots, Re$\gamma_1\geq$ Re$\gamma_2,$
 %we set $$\gamma_{1,2}=\frac{-a_1\mp\sqrt{a_1^2-4a_2}}{2}$$ (resp.).
 the corresponding $\gb_j$'s are given by (\ref{beta}) to be
 $$
\gb_1= \frac{1}{\gamma_2-\gamma_1}(b_1+\gamma_1b_2);
 \,\,
\gb_2=\frac{-1}{\gamma_2-\gamma_1}(b_1+\gamma_2b_2).
 $$
Thus, from Theorem \ref{nonsol2}, we immediately conclude the following
\begin{prop}
The adjoint operator $L^*$ for $L$ as in (\ref{secordex})
is not locally solvable if one the cases hold:
%$$
%\text{\rm Im}(\bar{\gamma_2}-\bar{\gamma_1})\text{ \rm Im}(b_1+\gamma_1b_2)
%>\text{\rm Re}(\bar{\gamma_2}-\bar{\gamma_1})
%\text{ Re}(b_1+\gamma_1b_2)$$
%$$
%\text{\rm Im}
%(\bar{\gamma_1}-\bar{\gamma_2})\text{ \rm Im }
%(b_1+\gamma_2b_2)>\text{Re}(\bar{\gamma_1}-\bar{\gamma_2})\text{ Re}(b_1+\gamma_2b_2)$$
\begin{itemize}
\item[1)]
Re${\gamma_1}$ and Re$\gamma_2$ are non-zero and have the same sign;
\item[2)]
{Re}$\gamma_1=0>$ {Re}$\gamma_2$ and
$$\text{\rm Im}({\gamma_1}-{\gamma_2})\text{ \rm Im}(b_1+\gamma_1b_2)
>\text{\rm Re}{\gamma_2}
\text{\rm Re}(b_1+\gamma_1b_2);$$
\item[3)]
Re$\gamma_2=0<$ {Re}$\gamma_1$ and
$$\text{\rm Im}
({\gamma_2}-{\gamma_1})\text{ \rm Im }
(b_1+\gamma_2b_2)<\text{\rm Re}{\gamma_1}\text{\rm Re}(b_1+\gamma_2b_2)$$
\end{itemize}
\end{prop}
%
%We find that it is not necessary for the operator $P_n(X,Y)$ to be non-solvable for
%$P(X,Y)$ to be non-solvable.
As an example, let us set $a_2=\ga$ $=0$ and let $a_1=-2\gl$ for some real $\gl\neq 0$.
The characteristic roots are $1$ and $2\gl$.
Here, $\cL_{\infty}$ is self- adjoint and bases $\vec{\psi}^{\pm}$ of ker$\cL_{\infty}$ are given:
$\psi_1^{\pm}$ $=1$, $\psi_2^{\pm}$ $=$ $\int_{\pm\infty}^te^{\gl s^2}ds$ (resp.) when
$\gl<0$; and, $\psi_1^{\pm}$ $=$ $\int_{0}^te^{\gl s^2}ds,$ $\psi_2^{\pm}=1$ for
$\gl>0.$ So, ker$\cL_{\infty}$$\bigcap$$\cS(\BR)$$\setminus\{0\}$ is empty for any
such $\gl$.
Thus, the associated operator $L^*$ is locally solvable when $b_1=b_2=c=0.$
However, the operator $L^*$ is not locally solvable for any $b_2$ and $c$ when Re$b_1>0$, although
the operator $P_2(X,Y)^*$ is locally solvable.

We also note, conversely, that non-solvability of $P_n(X,Y)$ is not generally a sufficient condition for
non-solvability of $P(X,Y)$.
A class of operators, known as generalized Laplacians,
serve to demonstrate this point: Let $L_{\gl,\ga}$ be given by
\begin{equation}\label{mprex}
(\gl^2-1)L_{\gl,\ga}=(1-\gl^2)X^2 + Y^2+{i\gl}(XY+YX) +i\ga[X,Y]
\end{equation}
for constant $\gl,\ga$ such that $-1<\gl<1$ is real and $\ga$ $\in\BC.$ Here,
$L_{\gl,\ga}$ is not locally solvable when $\ga\in\BZ^+$ is odd; yet, for any constant $c\neq0,$ $L_{\gl,\ga}+c$ is locally solvable
for any $\gl,\ga$ in their domains (see Theorem 3.3 \cite{mpr}).
We elaborate on operators (\ref{mprex}) as we sketch an alternate proof (viz \cite{dpr,fs,s})
that the operator $L_{\gl,\ga}+c$ is locally solvable for any $c\neq 0:$

We start with the case $\gl=0$.
Since the characteristic roots are $\pm 1$, from Theorem \ref{reg} it suffices to show that
$\cL_{\gm}$ $=\partial_t^2-t^2+\ga-z\gm^{-2}$ for any fixed $\ga$ and $z\neq0$ is regular since the analysis for each associated $\cL_{\gm}^{\pm}$
is essentially the same. In the case
$\ga>0$ not an odd integer we see that
ker$\cL_{+\infty}$ $=$ ker$\cL_{+\infty}^*$ which contains no $\cS(\BR)$-class functions other than the zero function.
It then follows from
Theorem \ref{solve}, that (\ref{mprex}) is locally solvable.

Let us now continue with
the case $\gl=0$ but now with $\ga>0$ an odd integer.
We form (admissible) bases $\vec{\psi}^{\pm}$ of ker$\cL_{\gm}$ using the well-known parabolic cylinder functions
$U,V$ (cf. Chapter 19 \cite{as}): We set $\fa=\fa(\gm,\ga,z)\df\frac{-\ga+z\gm^{-2}}{2}$ and write
\begin{align}\phi_1^{\pm}(t,\gm)&=V(\fa,\pm\sqrt{2}t)\asymp e^{t^2/2}(\pm t)^{\frac{-\ga+z\gm^{-2}-1}{2}}
\notag\\
\phi_2^{\pm}(t,\gm)&=U(\fa,\pm\sqrt{2}t)\asymp e^{-t^2/2}(\pm t)^{\frac{\ga-z\gm^{-2}-1}{2}}
\notag\end{align}
for $|\gm|>2|z|$, $\pm t>1$ (resp.).
With sufficiently large $\gm_0>2|z|,$
the associated transition matrix $A(\gm)$
satisfies\begin{align}
[A]_{1,1}&=\sin(\pi \fa)\asymp 1&; [A]_{1,2}&=\pi^{-1}\gC(\fa+1/2)\cos^2(\pi \fa)\asymp \gm;
\notag\\
[A]_{2,1}&=\pi/\gC(\fa+1/2)\asymp \gm &;[A]_{2,2}&=\sin(\pi \fa)\asymp 1\notag
\end{align}
on $(\gm_0,\infty);$ and, hence, $\cL_{\gm}$ is regular.

We elaborate further on the odd $\ga$ case above to give a partial demonstration of our parametrix,
regarding the steps involving (\ref{Fl}).
The estimates on the $\phi_j^{\pm}$ hold on complex sectors of the form
$|\pm\arg t$ $-\pi/4|$ $\leq \gd$ $< $ $\pi/4$
(resp.) for $\gm$ sufficiently large along complex arcs
Im$\gm\geq 0,$ $|\gm|$ fixed. Hence, we have clear choices for $\cU_{1}^{\pm}$ as in (\ref{tz}),
using $\ga_1=\pi/4$, $\gd$ $=\pi/2$, and $\gl_1=\gm_0^2$.
Moreover, there are constants $c_1,c_2\neq 0$ so that
$\cH_1^{\pm}=c_1\phi_2^{\pm}$ and $\cH_2^{\pm}=c_2\phi_1^{\pm}$ (resp.),
whereby the integral formula along with contours $\cC_l^{\pm}$ are readily determined.

We now suppose that $\gl\neq 0.$ The operation $\cP_{\gm,\gl}:$
$\cP_{\gm,\gl}\phi(t)$ $=$ $e^{\gm\gl t^2/2}\phi^{\pm}_j(t/\sqrt{1-\gl^2})$
gives
$\text{ker}\gL_{\gm}(L_{\gl,c})
$ $=\cP_{\gm,\gl}(\text{ker}(\gL_{\gm}(L_{0,c}))$
(see Proposition 7.2 \cite{mpr}). Moreover, $\forall$
$\phi\in$ ker$\cL_{\gm},$ $\cP_{\gm,\gl}(\phi)$
$\in\cS(\BR)$ if and only if $\phi\in\cS(\BR).$
Since $|\gl|<1$, the polynomial $P$
associated with (\ref{mprex}) has property $\cG.$
It now follows from Theorem \ref{solve}
that (\ref{mprex}) is locally solvable in the specified
domains of $\ga$ and $\gl$. 

%% file: appendixresub.tex
\begin{prop}\label{dunno}
Given $\ga\in \BR,$ and $a,\gamma>0,$
$$\int_0^te^{\gamma s^2 + \ga s}(1+s)^a\,dt
\lesssim e^{\gamma t^2 +\ga t}(1+t)^{a+1}$$
$$\int_t^{\infty}e^{-\gamma s^2 + \ga s}(1+s)^a\,dt\lesssim e^{-\gamma t^2 +\ga t}(1+t)^{a}
$$
for $t\geq 0.$ Moreover, for any positive $\ga_0$ constants
implicit in the asymptotics may be chosen so that the estimates
are uniform for $|\ga|\leq \ga_0.$
\end{prop}
\begin{proof} The results hold for $\ga=0$ \cite{w1}; so, employing the change of
variables $u=s+\frac{\ga}{2\gamma}$ we compute
\begin{align}
&\int_0^te^{\gamma s^2+\ga s}(1+s)^a\,ds=
\int_{\frac{\ga}{2\gamma}}^{t+\frac{\ga}{2\gamma}}e^{\gamma
u^2-\frac{\ga^2}{4\gamma^2}}
\left(1+|u-\frac{\ga}{2\gamma}|\right)^a\,du\notag\\
& \leq e^{-\frac{\ga^2}{4\gamma}}\left(1+|\frac{\ga}{2\gamma}|\right)^a
\int_{\frac{\ga}{2\gamma}}^{t+\frac{\ga}{2\gamma}}
 e^{\gamma u^2}\left(1+|u|\right)^a\,du\notag\\&
\leq e^{-\frac{\ga^2}{4\gamma}}\left(1+|\frac{\ga}{2\gamma}|\right)^a
[\int_{\frac{-|\ga_{_0}|}{2\gamma}}^0e^{u^2}(1+|u|)^a\,du +
%e^{-\frac{\ga^2}{4\gamma^2}}
\int_0^{t+\frac{|\ga|}{2\gamma}} e^{\gamma u^2}(1+u)^a\,du]
\notag\\
&\lesssim e^{-\frac{\ga^2}{4\gamma}}
\int_0^{t+\frac{|\ga|}{2\gamma}} e^{\gamma u^2}(1+u)^a\,du
\lesssim
e^{-\frac{\ga^2}{4\gamma}}e^{\gamma(t+\frac{\ga}{2\gamma})^2}
\left(1+t+\frac{|\ga|}{2\gamma}\right)^{a+1}\notag\\
&\lesssim e^{\gamma t^2 +\ga
t}\left(1+\frac{\ga_0}{2\gamma}\right)^{a+1}(1+t)^{a+1} \lesssim e^{\gamma
t^2 +\ga t}(1+t)^{a+1} \notag\end{align} which holds uniformly for
$t\geq 0 $ and for $|\ga|\leq \ga_0.$

By the same substitution we likewise compute for $t\geq
0,$
\begin{align}
&\int_t^{\infty}e^{-\gamma s^2+\ga s}(1+s)^a\,ds =
e^{\frac{\ga^2}{4\gamma}}\int_{t+\frac{\ga}{2\gamma}}^{\infty}e^{-\gamma
u^2}\left(1+|u-\frac{\ga}{2\gc}|\right)^a\,du
\notag\\
&\leq e^{\frac{\ga^2}{4\gamma}}\left(1+\frac{\ga_0}{2\gc}\right)^a
[\int_{-\frac{\ga_0}{2\gc}}^{\frac{\ga_0}{2\gc}}e^{\gc
u^2}(1+|u|)^a\,du+\int_{t+\frac{\ga}{2\gc}}^{\infty}
e^{\gc u^2}(1+|u|)^a\,du]\notag\\
&\lesssim
e^{\frac{\ga^2}{4\gamma}}e^{-\gamma(t-\frac{\ga}{2\gamma})^2}(1+t)^a
=e^{-\gamma t^2+\ga t}(1+t)^a\notag\end{align} uniformly for
$|\ga|\leq \ga_0$ and $t\geq 0;$ and, the result follows.
\end{proof}

\begin{remark}\label{rem1}
We make a correction to equation (2.13) \cite{w1}.\footnote
{The author thanks the anonymous reviewer of said article for having pointed out the error
in the original manuscript.}
For each $\gt>0$ the support of $\Breve{F}_{\gt}$ (the full Fourier transform of $F_{\gt}$)
contains a neighborhood of the origin and
thus the result should appear as
\begin{align}
||v_{\gt}||^2_{(-s)}&\leq c_s\gt^{-4}\int_{\xi_3\geq1} (1+(\xi_1\gt)^2+(\xi_2\gt)^2+(\xi_3\gt^2)^2)^{-s}|\Breve{F}_{\gt}(\vec{\xi})|^2d\vec{\xi}
\notag\\
&\leq c_s\gt^{-4-2s}\int_{\xi_3\geq 1}\xi_3^{-s}|\Breve{F}_{\gt}(\vec{\xi})|^2d\vec{\xi}
\leq c_s\gt^{-4-2s}\int_{\xi_3\geq 1}|\Breve{F}_{\gt}(\vec{\xi})|^2d\vec{\xi}\notag
\end{align}
for some constant $c_s>0.$  This follows by the Paley-Wiener Theorem and the $\cS(\BR)$-mode of convergence of
$\Breve{F}_{\gt}$
to its limit $\Breve{\Psi},$ supported in $\{\vec{\xi}|\xi_3\geq 1\}.$
\end{remark}

\begin{remark}\label{w3stuff}
The statements of Propositions 4.1, 4.4 and Lemma 4.5 \cite{w3} (here $z=\gm$)
each hold for operators $\cL_{\gm}$
of the present article under our definition of resolving permutations
but restricting the domain of $A$ to semi-infinite intervals $(R,\infty)$
(here $R=\gm_0$) for sufficiently large, fixed $R>0$.
\end{remark}
Let us elaborate on the above remark to point out modifications
necessary to apply the aforementioned results to the present work: We may interpret the
classifications $\cP_R,$ $\cF_R$ and $\cB_R$ in the manner of \cite{w3} (Appendix) regarding
only large $z>0$ (ignoring meromorphicity near the origin).
Furthermore, in the context of the present article,
we recast (19) \cite{w3} as follows:
If $\cL_{\gm}$ is regular, then there exists an admissible pair of bases for ker$\cL_{\gm}$
such that $\vec{\phi_l}^-=A(\gm) \vec{\phi_l}^+$ where
$A$ in an $n\times n$ matrix which likewise takes
the block form
\begin{equation}\label{decompIImtrx}
%\left(\begin{matrix}\phi_1^-(x,z)\\
%\vdots\\
%\phi_n^-(x,z)\end{matrix}\right) =
A= \left(
\begin{matrix}
0^{J^-\times J^-}
& M^{J^-\times J^+}&0
%^{(n-J^--J^+)\times(n-J^--J^+)}
\\
N^{J^+\times J^-}&0^{J^+\times J^+}
&0
\\
\cA&\cB& Q
\end{matrix}
\right)
.\end{equation}
Here, $Q=Q^{(n-J^+-J^-)\times(n-J^+-J^-)}$ and the dimensions of $\cA$ and $\cB$ are thus obvious.  Except for their dimensions,
the corresponding blocks have the same properties as those in Proposition 4.2 \cite{w3}.
Moreover, if $\cL_{\gm}$ is quasi-regular, there is a transition matrix $A(\gm)$
which takes such block forms on open sets $U_i$ (corresponding to various resolving $\gs_i$),
whereby equation (20) and Proposition 4.4 therein also hold.

Let us say that $A$ has a degenerate row when there is
an index $K > J^+$ whereby $[A(\gm_l)]_{j,k}:$ $k<K$ are each rapidly decreasing
on a common sequence $\gm_l$ $\rightarrow$ $\infty$ (srd). From the proofs of Corollaries
3.2 and 4.2 of \cite{w3} we find that if $\cL_{\gm}$ is irregular, in the present context,
then  $(A^{-1})^{\dagger}$ has a degenerate row and the result of Lemma 4.5 therein also follows here.
We note that the order of the admissible pairs according to $\pm$ sign is arbitrary.
Remark \ref{w3stuff} still holds upon interchange of bases,
redefining $A,$ $\gs$ and the $J^{\pm}$ accordingly.

{\it Proof of Proposition \ref{moremainests}:}
It follows from Section 4 of
\cite{c} that given the solution $v$ as in (\ref{w}), we may
find another solution by solving the equation
\begin{equation}
\tilde{v}^{\prime} = \tilde{\Lambda}\tilde{v} +
(\tilde{\cR}+\cC)\tilde{v}.
\label{red}
\end{equation}
Here $\tilde{\Lambda}$ and $\tilde{\cR}$ are obtained by
deleting the last row and last column from $\Lambda$
and $\cR$, respectively and where
$\cC$ is an $(n-1)$$\times$$(n-1)$ matrix with
components
$[\cC]_{i,j}
=\frac{v_i}{v_n}[B]_{n,j},
$
supposing (perhaps by increasing $t_0$ if necessary)
that $v_n(t)$ $\neq$ $0$ for $t\geq$ $t_0.$
Then, the vector
$$
w= g
\left(
\begin{matrix}
v_1\\
\vdots\\
v_{n-1}\\
v_n
\end{matrix}\right)
+
\left(
\begin{matrix}
\tilde{v}_1\\
\vdots\\
\tilde{v}_{n-1}\\
0
\end{matrix}
\right)
$$
yields a solution to $\cL_{\gm}y=0$
where
\begin{equation}
g^{\prime} = \frac{1}{v_n}(\sum_{j=1}^{n-1}[B]_{n,j}\tilde{v}_j)
.\label{gp}
\end{equation}

To obtain estimates for $w$, we have from analysis as in (\ref{v}) that
\begin{equation}\label{vtil1}
|\tilde{v}|\lesssim |e^{\Phi_{2}}|;\,\,
|\tilde{v}_{n-1}-e^{\Phi_2}|\lesssim |t^{-1}||e^{\Phi_{2}}|;\,\,
|\tilde{v}_j|\lesssim |t^{-1}||e^{\Phi_{2}}|
\end{equation}
for $j<n-1.$
Since the right-hand side of (\ref{gp}) is majorized by
$|{t^{-2}}||\exp(\Phi_{2}-\Phi_1)(t)|$
uniformly in $t$ and $\gm$,
we may uniquely define $g$ by
setting
$g(t_0)$ $=0$ to obtain
$$|g(t)|\lesssim |t^{-1}||\exp(\Phi_{2}-\Phi_1)(t)|.$$

It then follows that $w$ and its components also satisfy estimates
(\ref{vtil1}) for $j\leq n.$  We then obtain
another solution $\psi_{n-1}=[Sw]_{1,1}$, with $\gamma_2$
replacing $\gamma_1$ in $S$, this one satisfying
$$\dertn{j}\psi_{2}=[Sw]_{1,j}=(\gamma_{2}t+\gb_{\gm,n-1})^{j-1}
e^{\Phi_{2}(t,\gm)}(1+o(1))\,\,\text{\rm as}\,\,t\rightarrow
\infty$$ with bounds uniform for $\gm>\gm_0$. Moreover, from the
construction of $g$, we have that $\psi_{2}(t,\gm)$
is a holomorphic function of $\gm$ provided $t$ is
sufficiently large.  It is clear that $\psi_{2}(t,\gm)$ extends
to a holomorphic function of $\gm$ for all $t.$ We then
obtain estimates for linearly independent functions $\psi_j:$
$1\leq j\leq n$ by induction each being holomorphic functions of $\gm$.
Estimates and holomorphic properties for large $t<0$ follow
similarly: By the same analysis on $P(i\partial_s,s)$ with $s=-t$,
the $\gamma_j$'s remain the same; yet, by (\ref{D1}) and (\ref{beta}), the $\beta_j$'s change sign.
The main result, having been shown up to derivatives of order $j=n-1$,
follows inductively by applying (\ref{mainest}) to
$\dertn{l+n}\psi$ $=\dertn{l}\circ\cL_{\gm}\psi$ for any $\psi$ $\in$
ker$\cL_{\gm}$ (cf. \cite{w1}); here, $$\dertn{l}\circ\cL_{\gm}=
\sum_{j=0}^{n+l}a_{j}(t,\gm)\dertn{j}$$ whose coefficients satisfy
$a_j(t,\gm)\lesssim$ $(1+|t|)^{n+l-j}$.
\qed